\documentclass[11pt]{article}
\usepackage{amsmath,amssymb,amsfonts,amscd}
\usepackage{mathrsfs,slashed}
\usepackage{marginnote}
\usepackage{color}
\usepackage{bm}

\usepackage[all]{xy}
\usepackage{graphics}

\newtheorem{thm}{Theorem}[section]
\newtheorem{lem}[thm]{Lemma}
\newtheorem{prop}[thm]{Proposition}
\newtheorem{cor}[thm]{Corollary}
\newtheorem{defn}[thm]{Definition}

\newtheorem{example}[thm]{Example}

\numberwithin{equation}{section}

\DeclareMathOperator{\Hom}{Hom}
\DeclareMathOperator{\End}{End}

\DeclareMathOperator{\Ext}{Ext}

\DeclareMathOperator{\D}{D}
\DeclareMathOperator{\Tr}{Tr}
\DeclareMathOperator{\rad}{\rm rad}
\DeclareMathOperator{\soc}{\rm soc}

\newcommand{\pmodcat}[1]{#1\mbox{{\rm -proj}}}

\DeclareMathOperator{\idim}{injdim}
\DeclareMathOperator{\gdim}{gldim}
\DeclareMathOperator{\domdim}{domdim}
\DeclareMathOperator{\lmod}{\rm{-mod}}
\DeclareMathOperator{\add}{\rm{add}}

\DeclareMathOperator{\evd}{\rm{rd}}

\DeclareMathOperator{\rep}{repdim}
\DeclareMathOperator{\cf}{rigdim}

\DeclareMathOperator{\stmod}{\mbox{{\rm -{\underline{mod}}}}}
\newcommand{\lra}{\longrightarrow}
\newcommand{\lraf}[1]{\stackrel{#1}{\lra}}

\newenvironment{proof}{\noindent \textit{Proof.}\,}{\hfill$\square$ \vskip5pt}

\begin{document}

\title{Rigidity dimension - a homological dimension measuring resolutions of
  algebras by algebras of finite global dimension}
\author{Hongxing Chen\footnote{Corresponding author}, Ming Fang, Otto Kerner,\\
 Steffen Koenig and Kunio Yamagata}
 \date{}
 \maketitle

 \begin{abstract}
   A new homological dimension is introduced to measure the quality of
   resolutions of `singular' finite dimensional algebras (of infinite
   global dimension) by `regular' ones (of finite global dimension).
   Upper bounds are established in terms of extensions and of Hochschild
   cohomology, and finiteness in general is derived from homological
   conjectures. Then invariance under
   stable equivalences is shown to hold, with some exceptions when there
   are nodes in case of additive equivalences, and without exceptions in
   case of triangulated equivalences.
   Stable equivalences of Morita type and derived equivalences,
   both between self-injective algebras, are shown to
   preserve rigidity dimension as well.
 \end{abstract}

{\footnotesize\tableofcontents\label{contents}}

\section{Introduction}

A finite dimensional algebra $E$ may be called `regular' if it has finite
global dimension, that is, cohomology $\Ext_E^n(X,Y)$ between $E$-modules
vanishes from a certain degree $N_0$ on.
Given a `singular' algebra $A$, of infinite global
dimension, `resolving' it by $E$ means, in the weakest sense, that there
exists an idempotent $e=e^2 \in E$ such that $A = eEe$. By a construction
due to Auslander, this is always possible. When $A$ is
self-injective this implies that $eE$ is a generator-cogenerator in $E\lmod$;
in general one may require that. There always is a double centraliser property
between $E$ and $A=eEe$ on the module $eE$. Algebras $E$ `resolving' $A$ then
are the endomorphism rings of generator-cogenerators over $A$. Among the
algebras $E$ one has to single out those with finite global dimension.
In this sense, a classical or quantised Schur algebra is a resolution of
the group algebra of a symmetric group or a Hecke algebra, respectively.
In Rouquier's terminology \cite{Ro08},
the Schur algebra is a `quasi-hereditary' cover, where quasi-hereditary
refers to the additional structure of its module category being a highest
weight category.

An appropriate measure of the `quality' of the resolution is the
dominant dimension of $E$, as has been demonstrated in \cite{FK,FK14, F14,FM}
in the case of Schur algebras and more generally of gendo-symmetric
algebras. The choice of dominant dimension to measure
how close $E$ and $A$ are related (the closer, the larger the dominant
dimension of $E$ is) is compatible with Rouquier's setup,
and it reflects the role
dominant dimension is playing in Morita-Tachikawa correspondence, Auslander's
correspondence, Auslander's representation dimension and Iyama's higher
representation dimension \cite{Iyama07}.
There also are close relations to Iyama's maximal
orthogonal modules \cite{Iyama07b} and to Iyama and Wemyss' cluster tilting
theory for commutative algebras \cite{IW}.
\medskip

The aim of this article is to introduce a new
homological dimension, which provides
information on the `best possible' resolution of a fixed algebra $A$:

\[
\cf(A) := \sup\Big\{\domdim \End_A(M) \, \, \Big | \, \, \parbox{7cm}{$M$ is a
  generator-cogenerator in $A\lmod$
    and $\gdim\End_A(M)<\infty$}   \Big\}.
\]

By Theorem \ref{thm:Muller characterization} (M\"uller), when $M$ is a
generator-cogenerator in $A\lmod$, the dominant dimension $\domdim \End_A(M)$
of its endomorphism ring is determined by the rigidity degree
of $M$ (Definition \ref{EVD}),
which measures vanishing of self-extensions of $M$. Therefore, the
new dimension is called {\bf rigidity dimension} of $A$.
Note that non-semisimple algebras of
finite global dimension always have finite dominant dimension,
so the supremum is taken over a set of natural numbers, unless $A$ is
semisimple.

When $A$ has infinite global dimension, $\cf(A)$ measures how close $A$ can come
to an endomorphism ring $E=\End_A(M)$ of a generator-cogenerator $M$, with
$E$ having finite global dimension.
Semisimple algebras obviously have infinite
rigidity dimension. When $A$ is not self-injective, $\cf(A)$ still is defined,
but its meaning is less clear. We will see that always $\cf(A) \geq 2$.

A combination of global and dominant dimension also occurs in the definitions
of Auslander's representation dimension, which is always finite, and Iyama's
higher representation dimension, which is often infinite. It turns out that
rigidity dimension controls finiteness of higher representation dimension:
$\rep_n(A)$ is finite if and only if $\cf(A) \geq n+1$. Thus, rigidity dimension
can be viewed as a companion of higher representation dimension.

We will address two basic questions about this new dimension: {\em finiteness}
and {\em invariance under equivalences}.

Three approaches are developed to establish finiteness. The first approach
provides an upper bound for $\cf(A)$ in terms of the smallest degree
$n \geq 1$, if existent, for which $\Ext_A^n(\D(A),A)$ does not vanish (Theorem
\ref{injective}).
Here $\D$ is the usual $k$-duality over the ground field $k$.
This can be applied to maximal orthogonal modules and to
gendo-symmetric algebras, including Schur algebras of algebraic groups and
blocks of the Bernstein-Gelfand-Gelfand category $\mathcal O$ of semisimple
complex Lie algebras. For the latter algebras, this method implies
that further `resolving' a gendo-symmetric algebra $E$ of finite global
dimension, for instance a Schur algebra $S(r,r)$, which resolves a symmetric
algebra $A$ like  the group algebra $k \Sigma_r$ of a symmetric group,
cannot produce a better resolution of $A$, with respect to $\cf$.

The second approach provides an upper bound in terms of the
smallest positive degree of a non-nilpotent homogenous generator
of Hochschild cohomology (Theorem
\ref{thm:cf dim and Hochschild cohomology ring}). This implies finiteness
of rigidity dimension for all non-semisimple group algebras (Theorem
\ref{cor:cf dim is finite for group algebra}). Symonds' proof of
Benson's regularity conjecture then implies that the order of the group is
an explicit, but weak, upper bound for $\cf(kG)$.

The third approach derives finiteness of rigidity dimension
for all non-semisimple algebras from homological conjectures (Theorem
\ref{thm:cf finite from conjectures}): Assuming
Tachikawa's first conjecture yields finiteness for algebras that are
not self-injective, as an application of the first approach.
Finiteness in general is shown to follow from Yamagata's conjecture.
These conjectures are in general open;
for finite-dimensional algebras, no counterexamples are known. Tachikawa's
first conjecture is part of a reformulation of Nakayama's conjecture.
\medskip

Morita equivalences do preserve rigidity dimension, which is shown also
to be invariant under stable equivalences provided the algebras involved do
not have nodes (Theorem
\ref{thm:rigidity dimension is invariant under stable equivalence}).
In the presence of nodes (which does not happen frequently),
rigidity dimension can change, as we show by examples. Algebras without nodes
always have minimal rigidity dimension in their stable equivalence class
(Theorem \ref{thm:rigidity dimension is invariant under stable equivalence}).
Non-invariance disappears
when the equivalences preserve the triangulated structure which stable
categories of self-injective algebras are known to have (Theorem
\ref{Stable equivalences of self-injective algebras}).

Stable equivalences of Morita type, and thus also derived equivalences, between
self-injective algebras do leave rigidity dimension invariant. More precisely,
stable equivalences of adjoint type are shown always to preserve
rigidity dimension, for all finite dimensional algebras (Theorem
\ref{thm:rigidity dimension is invariant under stable equivalences of adjoint type}).

\medskip

The organisation of this article is as follows:

Section \ref{sec:rigidity dimension, introduction}
starts by recalling some known
material and then gives the definition of rigidity dimension, some basic
properties, and connections with higher representation dimension, maximal
orthogonal modules and weakly Calabi-Yau properties. Section
\ref{sec:finiteness} provides three approaches to proving finiteness of
rigidity dimension; first, an upper bound in terms of vanishing of extensions
between injective and projective modules, then an upper bound in terms of
degrees of generators of reduced Hochschild cohomology, and finally a
general proof of finiteness, for non-semisimple algebras, based on assuming
validity of homological conjectures. In Section
\ref{sec:stable invariant}, rigidity dimension is shown to be invariant
under stable equivalences between algebras without nodes. More generally,
algebras without nodes are shown to have minimal rigidity dimension in their
stable equivalence class. The inequalities may be proper, as we illustrate
by examples. Stronger results are shown for self-injective algebras; in
particular, invariance of rigidity dimension holds
when equivalences are required to preserve the triangulated structure.
In Section \ref{sec:stableqMorita}, stable
equivalences of adjoint type are shown to preserve rigidity dimension; this
implies invariance of rigidity dimension under stable equivalences of Morita
type between self-injective algebras, and hence also under derived equivalences
between self-injective algebras.
\smallskip

In the subsequent article \cite{CFKKY2}, rigidity dimensions of classes of
examples will be determined.
\bigskip

{\bf Notation.} Throughout this paper, $k$ is an arbitrary but fixed
field. Unless stated otherwise, all algebras are finite-dimensional
associative $k$-algebras with unit, and all modules are finite-dimensional
left modules. The set of positive integers is denoted by $\mathbb{N}$; the set
of non-negative integers is denoted by $\mathbb{N}_0$.

 Let $A$ be an algebra. By $A\lmod$, we denote the category of all
left $A$-modules.
 The syzygy and cosyzygy operators of $A\lmod$ are denoted by $\Omega_A$ and $\Omega_A^{-}$, respectively. Let $A^{\rm {op}}$ be the opposite algebra of $A$. Then $\D:=\Hom_k(-,k)$ is a duality between $A\lmod$ and $A^{\rm {op}}\lmod$.

 Let $\mathcal{X}$ be a class of $A$-modules. By $\add(\mathcal{X})$, we denote the smallest full subcategory of $A\lmod$ which contains $\mathcal{X}$ and is closed under finite direct sums and direct summands. When $\mathcal{X}$ consists of only one object $X$, we write $\add(X)$ for $\add(\mathcal{X})$. In particular, $\add({_A}A)$ is exactly the category of projective $A$-modules and also denoted by $A\pmodcat$.  Let $\mathscr{P}_A$ and $\mathscr{I}_A$ stand for the set of isomorphism classes of indecomposable projective and injective $A$-modules, respectively.

 Let $n$ be a natural number. We denote by $\mathcal{X}^{\perp_n}$ (respectively, ${^{\perp_n}}\!\mathcal{X}$) the full subcategory of $A\lmod$ consisting of modules $Y$ such that $\Ext_A^i(X,Y)=0$ (respectively, $\Ext_A^i(Y,X)=0$) for $1\leq i\leq n$ and $X\in\mathcal{X}$. Similarly, we write $X^{\perp_n}$ for $\mathcal{X}^{\perp_n}$ whenever $\mathcal{X}=\{X\}$.

 The module $M$ is called \emph{basic} if every indecomposable direct summand of $M$ is multiplicity-free. Moreover, the head, radical and socle of $M$ are denoted by $\mathrm{hd}(M)$, $\rad(M)$ and $\soc(M)$, respectively.
The composition of two morphisms $f: X \to Y$ and
$g: Y\to Z$ in $A\lmod$ is denoted by $fg: X \to Z$.
In this sense, $\Hom_A(X,Y)$ is an $\End_A(X)$-$\End_A(Y)$-bimodule.

\section{Rigidity dimension}\label{sec:rigidity dimension, introduction}

The main object of study in this article, rigidity dimension,
will be introduced in the
second subsection. Before, we recall the definitions of global and dominant
dimension and define the rigidity degree of a module.
The third subsection then
provides basic properties and examples. This includes the connection with
Iyama's higher representation dimension, the relation with maximal
orthogonal modules, which can be used to provide lower bounds for
rigidity dimension, and finally a result on rigidity dimension of weakly Calabi-Yau
self-injective algebras, which implies that preprojective algebras of
Dynkin type have rigidity dimension exactly three.

\subsection{Global and dominant dimension, and rigidity}
\label{subsec:elements on domdim and gldim}

 An $A$-module $M$ is called a \emph{generator} if $A \in \add(M)$,
 a \emph{cogenerator} if $\D(A_A)\in \add(M)$, a \emph{generator-cogenerator}
 if it is both a generator and a cogenerator.

 Dominant dimension was introduced by Nakayama and later systematically
 study by Tachikawa, Morita, M\"{u}ller, Yamagata and many others, see
 \cite{Mu68,Tac73,Y96}. See also \cite{F14,FK,FK11, FK14,KY1,KY2}
 for some recent developments partly motivating the present article.

 \begin{defn}\label{defn:dominant dimension}
   The {\em dominant dimension} of $A$, denoted by
   $\domdim A$, is the largest $t \in {\mathbb N}_0$,
or $\infty$, such that in a minimal injective resolution
   $$0\lra {_A}A\lra I^0\lra I^1\lra\cdots\lra I^{t-1}\lra I^t\lra\cdots$$
   all $I^i$ are projective for $0\leq i < t$.
 \end{defn}
 Note that $\domdim A = \domdim A^{op}$ (see \cite[Theorem 4]{Mu68}).
 If $\domdim A\geq 1$, then the injective envelope of ${_A}A$ is
 faithful and projective. If $\domdim A\geq 2$, then any faithful
 projective-injective
 $A$-module $P$ has the \emph{double centralizer property},
 that is,  with $\Lambda=\End_A(P)$  there is an equality
 $A\cong \End_{\Lambda^{\rm{op}}}(P)^{\rm {op}}$.
In this case, also $A\cong\End_{\Lambda}(\D(P)) $, and $\D(P)$ is a
generator-cogenerator in $\Lambda\lmod$.
In  general, for calculating dominant dimensions of endomorphism
algebras,  the following definition is useful.

 \begin{defn}\label{EVD}
   Let $M$ be an $A$-module. The \emph{rigidity degree} of $M$,
   denoted by $\evd(_AM)$, is the maximal $n \in {\mathbb N}_0$, or $\infty$,
     such that
   $\Ext_A^{i}(M, M)$ vanishes for all $1 \leq i \leq n$. \\
   In other words, \mbox{$\evd(_AM)\geq n$}
   if and only if $ \Ext^i_A(M,M)=0 $ for all
   $1\leq i\leq n$.
\end{defn}

 In this article, rigidity degrees of modules, measuring the vanishing of
 self-extensions, are of particular
 interest, due to their connections with both dominant dimensions and
 Hochschild cohomology rings.
The connection with dominant dimension is provided by a result due to M\"uller:

 \begin{thm} \textrm{\bf(M\"uller \cite[Lemma 3] {Mu68})}
 \label{thm:Muller characterization}
 Let $A$ be an algebra and $M$ a generator-cogenerator in
 $A\lmod$. Then
 $\domdim \End_A(M)=\evd(_AM)+2$.
 \end{thm}

\subsection{Definition of rigidity dimension}

Both Rouquier's theory of quasi-hereditary covers \cite{Ro08} and the
theory of non-commutative crepant resolutions due to Van den Bergh,
Iyama and Wemyss, see for instance \cite{Iyama07}, `resolve'
algebras of infinite global dimension by algebras of finite global dimension.
The quality of such a resolution can be measured by the dominant dimension
of the resolving algebra or by the rigidity degree of the
generator-cogenerator over the algebra being resolved. The homological
dimension to be defined now, aims at measuring the quality of the best
possible such resolution.

  \begin{defn} \label{defn:rigidity dimension}
   The rigidity dimension of an algebra $A$ is defined to be
   \[
    \cf(A) = \sup\Big\{\domdim \End_A(M)\, \Big | \, \parbox{7cm}{$M$ is a generator-cogenerator in $A\lmod$
    and $\gdim\End_A(M)<\infty$}   \Big\}.
   \]
 \end{defn}
 The construction in the proof of Iyama's finiteness theorem on
 representation dimension, \cite[Lemma 2.2]{Iyama03},
 ensures the existence of a generator-cogenerator $M$ in $A\lmod$
 with $\End_A(M)$ being quasi-hereditary and hence
 $\gdim\End_A(M)<\infty$. Thus, the supremum is taken over a non-empty set.

 The rigidity dimension of a semisimple algebra
 is always $\infty$, as the dominant dimension of semisimple algebras
 is $\infty$. In Section \ref{sec:finiteness},
 we will give criteria to check
 finiteness of rigidity dimension for many non-semisimple algebras, and we will
 prove finiteness
 for all non-semisimple algebras after assuming Tachikawa's conjectures
 and Yamagata's conjecture.

 Morita-Tachikawa correspondence includes the statement that endomorphism
 rings of generator-cogenerators have dominant dimension at least two.
 Thus, values of $\cf$ always are at least two.

 \begin{cor} \label{cor:lessequaltwo}
For any algebra $A$, $\cf A \geq 2$.
 \end{cor}

 This also follows from a reformulation of the definition, using M\"uller's
 Theorem \ref{thm:Muller characterization}:
 \[
 \cf(A) = \sup\Big\{\evd(_AM)\Big | \parbox{7cm}{$M$ is a
   generator-cogenerator in $A\lmod$ and $\gdim\End_A(M)<\infty$}   \Big\}+2.
   \]

\subsection{Basic properties, examples and connections}

 \begin{prop}\label{prop:rigidity dimension blockwise}
 Let $A$ and $B$ be algebras. Then
 \begin{itemize}
 \item[\rm(1)] $\cf(A)=\cf(A^{\rm{op}})$ and $\cf(A\times B) = \min\{\cf(A),
               \cf(B)\}$.
 \item[\rm(2)] If $A$ and $B$ are Morita equivalent, then $\cf(A)=\cf(B)$.
 \item[\rm(3)] If $k$ is perfect, then $\cf(A\otimes_kB)\geq \min\{\cf(A), \cf(B)\}.$
 \end{itemize}
 \end{prop}
%

 \begin{proof}
 (1) and (2) are consequences of well-known facts:
 Both global dimension and dominant dimension are invariant
 under taking opposite algebras or passing to Morita equivalent
 algebras; the dominant dimension (respectively, global dimension)
 of the product of two algebras is the minimum
 (respectively, the maximum) of their dominant dimensions
 (respectively, global dimensions).

 $(3)$ Let $X$ and $Y$ be generator-cogenerators in $A\lmod$ and
 $B\lmod$, respectively. Then $X\otimes_kY$ is a generator-cogenerator in
 $(A\otimes_kB)\lmod$.
 Now $\End_{A\otimes_kB}(X\otimes_kY)\cong \End_A(X)\otimes_k\End_B(Y)$ as
 $k$-algebras, and
 $\domdim(\End_A(X)\otimes_k \End_B(Y))=\min\{\domdim\End_A(X),
 \domdim\End_B(Y)\}$ by \cite[Lemma 6]{Mu68}, and
 $\gdim(\End_A(X)\otimes_k\End_B(Y))= \gdim\End_A(X)+\gdim\End_B(Y)$
 whenever $k$ is a perfect field. Therefore,
 $\cf(A\otimes_k B)\geq \min\{\cf(A),\cf(B)\}$.
 \end{proof}

 \subsubsection{Relation with higher representation dimension}
 The following result exhibits rigidity dimension as a counterpart of the higher
 representation dimension introduced by Iyama \cite[Definition 5.4]{Iyama07}.
 Recall that for a natural number $n$, the $n$-th representation dimension
$\rep_n(A)$ of an algebra $A$ is defined to be
 \[
    \rep_n(A) =\inf\Big\{\gdim \End_A(M)\Big | \parbox{7cm}{$M$ is a generator-cogenerator in $A\lmod$ and
    $\domdim \End_A(M)\geq n+1$}\Big\}
 \]

 Auslander's classical representation dimension is $\rep_1$. As Iyama has
 shown \cite{Iyama03},
 $\rep_1$ is always finite. For $n \geq 2$, infinite values do occur.

 \begin{prop}\label{prop:relation between cf dim and rep dim}
 Let $A$ be an algebra and $n$ a positive integer.
 Then $\rep_n(A)<\infty$ if and only if $\cf(A)\geq n+1$.
 In particular, let $M$ be a generator-cogenerator in $A\lmod$ with
 $\evd(M)\geq  {\cf(A)-1}$. Then $\gdim\End_A(M)=\infty$ unless $\End_A(M)$ is semisimple.
 \end{prop}
 \begin{proof}
 If $\rep_n(A)<\infty$, then $\cf(A)\geq n+1$ by definition. Conversely,
 if $\cf(A)\geq n+1$, then there exists
 a generator-cogenerator $M$ in $A\lmod$ such
 that $\domdim \End_A(M)\geq n+1$ and $\gdim\End_A(M)<\infty$.
 Hence, $\rep_n(A)<\infty$.
 \end{proof}

 When $n=1$, the statement $\cf(A) \geq 2$ for all $A$ in Corollary
 \ref{cor:lessequaltwo}
 is a reformulation of Iyama's finiteness result, which has
 been used in the definition of rigidity dimension.

\begin{example}
 Let $A$ be a finite dimensional non-simple
 self-injective local $k$-algebra with $\mathrm{rad}^3(A)=0$. Then
 $\cf(A)= 2$, since by \cite[Theorem 3.4]{Ho84} every non-projective
 $A$-module has non-trivial self-extensions.
\end{example}


\subsubsection{Relation with maximal orthogonal modules, and a lower
bound for rigidity dimension}

Recall that Iyama \cite{Iyama07b} has defined a module $M$ to be {\em maximal
  $n$-orthogonal} (for $n \geq 1$) if it satisfies
$M^{\perp_n} = {}^{\perp_n} M = {\add}(M)$.
Such a module $M$ automatically is a generator-cogenerator.
Endomorphism rings of maximal $n$-orthogonal modules always have finite global
dimension.

 \begin{prop}\label{maximal-orthogonal}
 Let $A$ be a non-semisimple algebra and $n$ a natural number.
 If there exists a maximal $n$-orthogonal $A$-module, then $\cf(A)\geq n+2$.
 \end{prop}

 \begin{proof}
   Let $_AM$ be a maximal $n$-orthogonal $A$-module. By
   \cite[Theorem 0.2]{Iyama07},
 $M$ is a generator-cogenerator with $\evd(M)=n$ and $\gdim \End_A(M)<\infty$.
 Therefore, $\cf(A)\geq n+2$.
 \end{proof}

{\bf Remark.} Using Theorem \ref{injective}(2) below, if furthermore
either $1\leq \idim(_AA)\leq n+1$ or $1\leq \idim(A_A)\leq n+1$,
then $\cf(A)=n+2$. Indeed, under this assumption $\cf(A)\leq n+2$ by Theorem \ref{injective}(2),  and therefore $\cf(A)=n+2$.

\subsubsection{Weakly Calabi-Yau self-injective algebras}

Let $A$ be a self-injective algebra. The stable module category $A\stmod$ of $A$ is a $k$-linear Hom-finite triangulated category, and its shift functor $\Sigma$ is the cosyzygy functor $\Omega_A^{-1}$ \cite[Section 2.6]{Happel}.
Recall that $A\stmod$ is said to be \emph{weakly $n$-Calabi-Yau} for a natural number $n$ if there are natural $k$-linear isomorphisms
$$
\underline{\Hom}_A(Y,\Sigma^n(X))\cong \D\underline{\Hom}_A(X,Y)
$$
for any $X,Y\in A\stmod$. Since $A\stmod$ has a Serre duality $\Omega_A\nu_A$
\cite[Prop. 1.2]{ES}, $A$ is weakly  $n$-Calabi-Yau if and only if
$\Omega_A^{-n}$ and $\Omega_A\nu_A$ are naturally isomorphic
as auto-equivalences of $A\stmod$. Equivalently, $\Omega_A^{n+1}\nu_A$ is
naturally isomorphic to the identity functor of $A\stmod$.
If $A$ is symmetric, then it is weakly  $n$-Calabi-Yau if and only if
$\Omega_A^{n+1}$ is naturally isomorphic to the identity functor of $A\stmod$.

\begin{prop}\label{WCY}
Let $A$ be a non-semisimple self-injective algebra. If $A\stmod$ is weakly $n$-Calabi-Yau with $n\ge 1$, then $\cf(A)\leq n+1$.
\end{prop}

\begin{proof}
Since $A$ is self-injective and $n\geq 1$, we have $\Ext_A^{n}(X,X)\cong \underline{\Hom}_A(X, \Omega_A^{-n}(X))$ for any $A$-module $X$.
Recall that the shift functor $\Sigma$ of $A\stmod$ is given by the cosyzygy functor $\Omega_A^{-1}$. This yields $\underline{\Hom}_A(X, \Omega_A^{-n}(X))=\underline{\Hom}_A(X, \Sigma^n(X))$.
Since $A\stmod$ is weakly $n$-Calabi-Yau, it follows that $\underline{\Hom}_A(X, \Sigma^n(X))\cong \D \underline{\Hom}_A(X, X)$. Thus $\Ext_A^{n}(X,X)\cong \D \underline{\Hom}_A(X, X)$.
This implies that if $X$ is non-projective, then $\Ext_A^{n}(X,X)$ does not vanish. In this case, $\evd(_AX)\leq n-1$. Now, it is clear that $\cf(A)\leq n+1$.
\end{proof}

 \begin{cor}
Let $A$ be a preprojective algebra of Dynkin type over an algebraically closed
field. Then $\cf(A)=3$.
 \end{cor}

\begin{proof}
It is known that $A\stmod$ is weakly $2$-Calabi-Yau. Proposition \ref{WCY}
implies $\cf(A)\leq 3$.
Further, by \cite[Theorem 2.2 and Corollary 2.3]{GLS}
there exists a maximal $1$-orthogonal $A$-module. It follows from
Proposition \ref{maximal-orthogonal} that
$\cf(A)\geq 3$. Thus $\cf(A)=3$.
\end{proof}

\section{Finiteness} \label{sec:finiteness}

When defining a homological dimension, a basic question is: On which algebras
does it take finite values?
Semisimple algebras have infinite rigidity dimension,
for trivial reasons, which distinguish them from all other algebras. In the
first two subsections we provide two methods to prove finiteness. The first
one is using extension groups between injective and projective modules; this
works for algebras of finite global dimension and for gendo-symmetric
algebras. The second one is using Hochschild cohomology; this works for group
algebras of finite groups. The third subsection then derives
finiteness in general from (still unproven) homological conjectures due to
Tachikawa and Yamagata.

\subsection{Finiteness I: Relation with the extension groups
  $\Ext^*_A(\D(A),A)$}

 As $\D(A)\oplus A$ appears as a direct summand of every
 generator-cogenerator $M$ (up to multiplicities), the
 groups $\Ext^*_A(\D(A),A)$ naturally occur in the computation of
 the Yoneda algebra $\Ext^*_A(M,M)$, and therefore in the computation of
 the rigidity dimension of $A$.

 \begin{thm}\label{injective}
 Let $A$ be a non-selfinjective $k$-algebra. Then
 \begin{itemize}
 \item[\rm(1)] $ \cf(A)\leq \sup \{ n \in {\mathbb N}_0
\mid \Ext_A^{j}(\D(A), A) = 0
   \, \, \mathrm{for} \, \, 1 \leq j \leq n \}+2$.
  Equality holds if the endomorphism algebra of $A\oplus \D(A)$
  has finite global dimension.
  \item[\rm(2)] $\cf(A)\leq \idim ({_A}A)+1 \leq \gdim(A) +1$.
     \end{itemize}
 \end{thm}

 \begin{proof}
   (1) Let $d:=\sup \{ n\in\mathbb{N}_0 \mid \Ext_A^{j}(\D(A), A) = 0
   \, \, \mathrm{for} \, \, 1 \leq j \leq n \}$.
 Then, $d=\evd(_AA\oplus \D(A))$ since $\Ext^i_A(\D(A),A)\cong
 \Ext^i_A(A\oplus\D(A), A\oplus \D(A))$ for any $i\geq 1$.  Now let $M$ be a
 generator-cogenerator in $A\lmod$. As $A\oplus \D(A)\in \add(_AM)$,
 it follows that $\evd(_AM)\leq \evd(_AA\oplus \D(A))$ and therefore
 $\cf(A)\leq d+2$.
 If furthermore $\gdim \End_A(A\oplus \D(A))<\infty$, then
 $\cf(A) =\evd(_AA\oplus \D(A))+2=d+2$.

 (2) Let $m:=\idim ({_A}A)$. Then $m\geq 1$ since otherwise, $A$ is
 self-injective. When $m$ is infinite, there is nothing to show. Suppose
 $m$ is finite. Then $\Ext^m_A(\D(A),A)\neq 0$ and $(1)$ implies
 $\cf(A)\leq m-1+2=m+1$.
 \end{proof}

\begin{example}
 Let $A$ be a non-semisimple hereditary algebra.
 Then $\cf(A)=2$ since for any generator-cogenerator $M$ in
 $A\lmod$, we have $\D(A)\oplus A\in \add(M)$ and
 $\Ext^1_A(\D(A),A)\neq 0$.
\end{example}

 \begin{cor}\label{gendo-symmetric}
 If $A$ is a gendo-symmetric algebra, i.e., the endomorphism algebra of
 a generator over some symmetric algebra, then $\cf(A)\leq \domdim(A)$.
 \end{cor}

 \begin{proof}
 By \cite[Proposition 3.3]{FK11}, the dominant dimension of $A$ is at least two,
 and equals $$\sup\{ n \in {\mathbb N}_0 \mid \Ext_A^i(\D(A), A)=0,
 \, 1\leq i\leq n \}+2.$$
 This implies $\cf(A)\leq \domdim(A)$ by Theorem \ref{injective}(1).
 \end{proof}

 \begin{example}
 In particular, a gendo-symmetric algebra
 $A$ with $\domdim(A) = 2$ satifies $\cf(A) = 2$. Examples of such algebras
 are the non-simple blocks of the Bernstein-Gelfand-Gelfand category $\mathcal
 O$ of semisimple complex Lie algebras. The global dimension of these
 algebras always is an even number, which can be arbitrarily large.
 \end{example}

\subsection{Finiteness II: Relation with Hochschild cohomology}
Let $A$ be a $k$-algebra and $A^{\rm op}$ its opposite
algebra. The Hochschild cohomology ring $\mathrm{HH}^*(A)$ is the
  Yoneda extension algebra
 $\Ext^*_{A^{\rm ev}}(A,A)$ where $A^{\rm ev}=A\otimes_k A^{\rm op}$, the
  enveloping algebra of $A$. So, $\mathrm{HH}^*(A)$ is the direct sum of
  $\mathrm{HH}^i(A):=\Ext_{A^{\rm ev}}^i(A, A)$ for $i\in\mathbb{N}_0$.
This is an
  $\mathbb{N}_0$-graded $k$-algebra. In general, it is not commutative, but
  graded commutative. Moreover,  it may be infinite-dimensional as a vector
  space over $k$. However, if $A^{\rm ev}$ has finite global dimension, then
  $\mathrm{HH}^*(A)$ is finite-dimensional.
  The {\em reduced Hochschild cohomology ring}
 $\overline{\mathrm{HH}}^*(A)$ is the quotient
 $\mathrm{HH}^*(A)/\mathcal{N}$ where $\mathcal{N}$ is the
 ideal of $\mathrm{HH}^*(A)$ consisting of nilpotent elements.
 Although the problem whether the (reduced) Hochschild cohomology
 ring is finitely generated has been widely studied, very little seems to be
 known about the degrees of the homogeneous generators. We will show that
 rigidity dimension of $A$ is closely
 related to the minimal degree of non-nilpotent homogeneous generators of
 positive degree.

We will use the following result, which is essentially combining
\cite[Theorems 2.13 and  5.9] {B} in our situation.
Note that a main tool in \cite{B} is the {\em grade} defined by
Auslander and Bridger,
which is closely related to dominant dimension.

\begin{thm} ({\bf Buchweitz}) \label{Hochschild Cohomology}
Let $M$ be a generator-cogenerator in $A\lmod$ and let  $E=\End_A(M)$.
Then there is an $\mathbb{N}_0$-graded algebra homomorphism
$\mathrm{HH}^*(E)\to\mathrm{HH}^*(A)$ such that
$\mathrm{HH}^i(E)\to \mathrm{HH}^i(A)$ is an isomorphism for each
$0\leq i\leq \evd(_AM)$.
\end{thm}


The connection with rigidity dimension is as follows:

 \begin{thm}\label{thm:cf dim and Hochschild cohomology ring}
   Let $A$ be an algebra over a perfect field $k$.
 Suppose that $\overline{\mathrm{HH}}^{\ast}(A)$ is not concentrated
 in degree zero.  Then $\cf(A)$ is finite. \\
More precisely,
 let $\delta(A):=\inf\{i\geq 1\mid \overline{\mathrm{HH}}^i(A)\neq 0\}$.
 Then $\cf(A)\leq \delta(A)+1$.
 If $k$ has characteristic different from $2$ and $\cf(A)$ is even, then
 $\cf(A)\leq \delta(A)$.
 \end{thm}

 \begin{proof}
 Let $M$ be a generator-cogenerator in $A\lmod$ and  let $E=\End_A(M)$.
 By Theorem \ref{Hochschild Cohomology}, there is a graded algebra
homomorphism $\varphi: \mathrm{HH}^*(E)\to\mathrm{HH}^*(A)$ such that
$\varphi_i: \mathrm{HH}^i(E)\to \mathrm{HH}^i(A)$ is an isomorphism for
 $0\leq i\leq \evd(_AM)=\domdim E-2$, where the final equality follows from
 M\"uller's Theorem \ref{thm:Muller characterization}.

 If $\gdim E<\infty$, then $\mathrm{HH}^*(E)$ is a finite dimensional
 $k$-algebra,  and therefore all homogeneous elements in
 $\mathrm{HH}^*(E)$ of positive degrees are nilpotent. Using $\varphi$
 and $\varphi_m$, all homogenous elements in $\mathrm{HH}^*(A)$ of degree
 $m$ with $1\leq m\leq\domdim E-2$ are seen to be nilpotent.
 Hence, \mbox{$\domdim E-1\leq \delta(A)$} since $\delta(A)$ detects
 the minimal degree of homogeneous generators of positive degrees in
 $\overline{\mathrm{HH}}^*(A)$. Thus $\cf(A)\leq \delta(A)+1$.

 If $k$ has characteristic different from $2$, then homogeneous elements in
 $\mathrm{HH}^*(A)$ of odd degrees are nilpotent, and therefore $\delta(A)$ must
 be an even number or $\infty$. When $\cf(A)$ is an even number, then
 $\cf(A)\leq \delta(A)$ as claimed.
 \end{proof}

 \noindent {\bf Remark.} (1) Theorem \ref{thm:cf dim and Hochschild cohomology
 ring} provides an upper bound for the rigidity dimension of $A$ as long as some
 information on its Hochschild cohomology ring is known. Conversely, a lower
 bound for rigidity dimension is a lower bound for the degree
 of homogeneous generators in the reduced Hochschild cohomology ring.
 Examples in \cite{CFKKY2} will show that the bound in Theorem
 \ref{thm:cf dim and Hochschild cohomology ring} is optimal.

 (2) The reduced Hochschild cohomology ring
 $\overline{\mathrm{HH}}^*(A)$ is not finitely generated in general. The
 first example has been a seven dimensional $k$-algebra $A$
 found by Xu \cite{Xu08}.
 If $k$ has characteristic two, then $\delta(A)=1$ \cite{Xu08}, see also
 \cite[Theorem 4.5]{Snashall}, and therefore $\cf(A)=2$.

 (3) Even when $A$ has infinite global dimension, the reduced Hochschild
 cohomology ring  $\overline{\mathrm{HH}}^*(A)$ can be concentrated in degree
 zero, see \cite{BGMS,GSS} for examples.
 \medskip

 Theorem \ref{thm:cf dim and Hochschild cohomology ring} can be applied to
 group algebras, for which we obtain finiteness in all non-semisimple cases.

 \begin{thm}\label{cor:cf dim is finite for group algebra}
 Let $G$ be a finite group and $k$ a perfect field of characteristic $p\geq 0$.

 Then $\cf(kG) < \infty$ if and only if $p$ divides the order $|G|$ of $G$ (if
 and only if $kG$ is not semisimple). In this case, $\cf(kG) \leq |G|$.
 \end{thm}

 \begin{proof}
 By \cite[Proposition 4.5]{Link99}, there is an injective graded $k$-algebra
 homomorphism
 \[
 \theta : \mathrm{H}^*(G,k)\to \mathrm{HH}^*(kG)
 \]
 where $\mathrm{H}^*(G,k)=\Ext^*_{kG}(k,k)$ is the cohomology ring
 of the group algebra $kG$. Let $\overline{\mathrm{H}}^*(G,k)$ be
 the cohomology ring $\mathrm{H}^*(G,k)$ modulo nilpotent elements. Then
 $\theta$ induces an injective morphism $\overline{\mathrm{H}}^*(G,k)\to
 \overline{\mathrm{HH}}^*(kG)$.

 If $p$ does not divide $|G|$, then $kG$ is semisimple and hence
 $\cf(kG)=\infty$. If $p$ divides $|G|$, then by
Chouinard's theorem in group cohomology, see
\cite[Lemma 5.2.3 and Theorem 5.2.4]{Ben}, $k$ is not projective and
\[
 \gamma(kG):=\inf\{i\geq 1\mid \overline{\mathrm{H}}^i(G,k)\neq 0\} < \infty.
 \]

 The embedding $\overline{\mathrm{H}}^*(G,k)\to \overline{\mathrm{HH}}^*(kG)$
 gives $\gamma(kG)\geq \delta(kG)$, the minimal
 degree of homogeneous generators of $\overline{\mathrm{HH}}^*(kG)$ of
 positive degrees. Using Theorem
 \ref{thm:cf dim and Hochschild cohomology ring},
 it follows that $$\cf(kG)\leq \delta(kG)+1\leq \gamma(kG)+1 < \infty.$$

 The explicit upper bound follows from a major result by Symonds,
 who proved Benson's regularity conjecture. According to
\cite[Proposition 0.3]{Sym10}, which is a consequence of the main result in
\cite{Sym10}, for a finite group $G$ with more than one element, group
cohomology $\mathrm{H}^*(G,k)$ has a set of homogeneous generators of degree
at most $|G| - 1$. As a consequence, the reduced
 group cohomology ring has a set of generators in degrees smaller than
 $|G|$, and thus
 \[
 \gamma(kG):=\inf\{i\geq 1\mid \overline{\mathrm{H}}^i(G,k)\neq 0\}\leq |G|-1.
 \]
 \end{proof}

 \noindent {\bf Remark.} (1) Symonds' result \cite{Sym10} provides an
 upper bound for the degrees of homogeneous generators of $\mathrm{H}^*(G,k)$,
 Theorem \ref{cor:cf dim is finite for group algebra} shows that rigidity dimension
 of the group algebra $kG$ may be used to provide a lower bound for the
 degrees of homogeneous generators of the Hochschild cohomology ring modulo
 nilpotents.

 (2) In \cite{CFKKY2}, the rigidity dimensions of defect
one blocks of group algebras will be determined.

\subsection{Finiteness III: Using homological conjectures}
Is it to be expected that $\cf$
is always finite, except for semisimple algebras?
Some evidence is provided here by deriving this statement from some
homological conjectures.
\medskip

\begin{itemize}
\item[(TC1)] {\bf Tachikawa's First Conjecture} \cite[p.\! 115]{Tac73}
Let $A$ be a finite dimensional $k$-algebra. Suppose
$\Ext^i_A(\D(A),A)=0$ for all $i\geq 1$. Then $A$ is self-injective.

\item[(TC2)] {\bf Tachikawa's Second Conjecture} \cite[p.\! 116]{Tac73}
Let $A$ be a finite dimensional self-injective $k$-algebra and $M$ a
finitely
 generated $A$-module. Suppose $\Ext^i_A(M,M)=0$ for all $i\geq 1$.
 Then $M$ is a projective $A$-module.

\item[(YC)] {\bf Yamagata's Conjecture} \cite[p.\! 876]{SkowYam}
  There exists a function
 $\varphi: \mathbb{N}\to \mathbb{N}$ such that for any finite
 dimensional $k$-algebra $A$ with finite dominant dimension,
 $\domdim A\leq \varphi(n)$ where $n$ is the number of isomorphism classes
 of simple $A$-modules.
\end{itemize}

 (TC1) and (TC2) together are equivalent to Nakayama's conjecture.

\begin{thm}\label{thm:cf finite from conjectures}
 Let $A$ be a finite dimensional non-semisimple $k$-algebra.
 \begin{itemize}
   \item[\rm(1)] If $A$ is not self-injective, then (TC1)
   implies $\cf(A)<\infty$.
   \item[\rm(2)] If (YC) holds, then $\cf(A)<\infty$.
 \end{itemize}
 \end{thm}

 \begin{proof}
 (1) If $A$ is not self-injective, then by (TC1),
 $\Ext^n_A(\D(A),A)\neq 0$ for some natural number $n\geq 1$. Therefore, by
 Theorem \ref{injective}, $\cf(A)\leq n+2$.

(2) Let $M$ be a generator-cogenerator in $A\lmod$. Up to multiplicities of direct summands, suppose
 ${_A}M=A\oplus \D(A)\oplus \bigoplus_{i=1}^{m}M_i$, where
 $m\geq 1$, and $M_i$ is either zero or indecomposable, non-projective and non-injective.

{\it Claim.}  $\evd(M)=\min\{\evd(A\oplus \D(A)\oplus M_i\oplus M_j)\mid 1\leq i,j\leq m\}.$

 {\it Proof of claim.} For $1\leq i,j\leq m$, set $M_{i,j}:=A\oplus \D(A)\oplus M_i\oplus M_j$. Since $M_{i,j}\in \add(M)$,
 $\evd(M)\leq\evd(M_{i,j})$. If $\evd(M)$ is infinite, then there is nothing to show.
 Suppose $\evd(M)=s<\infty$. Then $\Ext_A^t(M,M)=0$ for $1\leq t\leq s$ and $\Ext_A^{s+1}(M,M)\neq 0$.
 Since $\add(M)=\add\big(\bigoplus_{1\leq i\leq m}(A\oplus \D{A}\oplus M_i)\big)$, there is a pair $(u,v)$ of integers with $1\leq u,v\leq m$ such that $\Ext_A^{s+1}(A\oplus\D{A}\oplus M_u, A\oplus \D{A}\oplus M_v)\neq 0$.
 This implies $\evd(M_{u,v})\leq s$, and thus $\evd(M_{u,v})=s$.

Suppose $E:=\End_A(M)$ has finite global dimension.
By assumption, $A$ is not semisimple, hence $E$ is not semisimple either.
It follows that $\domdim E\leq\gdim E<\infty$.
Moreover, by Theorem \ref{thm:Muller characterization}, $\domdim E=\evd(M)+2$ and $\domdim \End_A(M_{u,v})=\evd(M_{u,v})+2$.
Then the equality $\evd(M_{u,v})=\evd(M)$ implies that $\domdim \End_A(M_{u,v})=\domdim E<\infty$.
Denote by $d$ the number of isomorphism classes of indecomposable objects in $\add(A\oplus \D(A))$. Then
$M_{u,v}$ has at least $d$ and at most $d+2$ non-isomorphic indecomposable direct summands.
Applying $(YC)$ to $\End_A(M_{u,v})$ yields the inequality
$$\domdim\End_A(M_{u,v})\leq\max\{\varphi(d),\varphi(d+1),\varphi(d+2)\}.$$
Consequently, $\domdim E$ has a uniform upper bound,
which only depends on the function $\varphi$ and $d$. Thus $\cf(A)<\infty$.
 \end{proof}

\section{Stable equivalences and invariance I}
\label{sec:stable invariant}
 In this and the next section, we discuss the invariance of rigidity dimension
 under stable or derived equivalences. The first main result in this section
 shows that algebras without nodes have minimal rigidity dimension in their
 stable equivalence class. In particular, two stably equivalent algebras
 without nodes have the same rigidity dimension. The second main result
provides stronger information for stable equivalences between self-injective
algebras. In particular, stable equivalences preserving the triangulated
structure are shown to preserve rigidity dimension.

 \subsection{Notation and definitions, and some
correspondences}
 Let $A$ be an algebra. By $A\stmod$ we denote the stable
 module category of $A$. It has the same objects as $A\lmod$,
 but the morphism set between two $A$-modules $X$ and $Y$
 is given by
 $$\underline{\Hom}(X,Y):=\Hom_A(X,Y)/\mathscr{P}(X,Y)$$
 where $\mathscr{P}(X,Y)$ consists of homomorphisms
 factoring through projective $A$-modules.

We will use functor categories to translate from stable equivalences to
equivalences of abelian categories.
 The category of finitely presented functors from $(A\stmod)^{\rm{op}}$
 to the category $\mathscr{Ab}$ of abelian groups is
 denoted by $(A\stmod)\lmod$.
 Recall that a functor $F:(A\stmod)^{\rm{op}}\to \mathscr{Ab}$
 is finitely presented if there is a morphism
 $f:X\to Y$ in $A\stmod$ inducing an exact sequence
 of functors $\underline{\Hom}_A(-,X)\to\underline{\Hom}_A(-,Y)\to F\to 0$.
 Note that $(A\stmod)\lmod$ is an abelian category
 with enough projective objects
 and injective objects. For each $X\in A\lmod$,
 $\underline{\Hom}_A(-,X)$ is a projective object and  $\Ext^1_A(-,X)$ is
 an injective object; all projective (respectively, injective) objects
 are of these forms (for more details, see \cite{AR74,AR78})

 Let $A\lmod_{\mathscr{P}}$ and $A\lmod_{\mathscr{I}}$ be the full
 subcategories of $A\lmod$ consisting of all modules without
 projective and injective direct summands, respectively.
 Then there is a one-to-one correspondence between indecomposable
 projective (respectively, injective) objects in $(A\stmod)\lmod$
 and indecomposable modules in $A\lmod_{\mathscr{P}}$
 (respectively, $A\lmod_{\mathscr{I}}$).

 \begin{defn}\label{defn:varous stable equivalences}
   Two algebras $A$ and $B$ are \emph{stably equivalent} if $A\stmod$ and
   $B\stmod$ are equivalent as additive categories; equivalently,
   $(A\stmod)\lmod$ and $(B\stmod)\lmod$ are equivalent as abelian categories.
\end{defn}

 The main complication when studying invariance of homological dimensions
 under stable equivalences, comes from a particular class of modules called
 nodes:

\begin{defn}\label{defn:nodes in stable equivalences}
   An indecomposable $A$-module $S$ is called a
 \emph{node} if it is neither projective nor injective, and there
  is an almost split sequence $0\to S\to P\to T\to 0$ with $P$
   a projective $A$-module.
 \end{defn}
 A node $S$ must be simple, see, for example, \cite[V.Theorem 3.3]{ARS95}.
A node does not occur as a composition factor of $\rad(A)/\soc(A)$.

 Two stably equivalent algebras without nodes and without
 semi-simple summands share many homological invariants, such as
 stable Grothendieck groups, global dimensions and dominant dimensions
(see \cite{MV90}). Such stable equivalences do preserve projective dimensions.
 However, if one algebra has a node, simple examples of stable equivalences
 already show that projective and global dimensions are not preserved.
\medskip

Self-injective algebras with nodes are quite well-known (see \cite{ARS95}):
When a self-injective connected algebra $A$ has a node $S$, there exists
an indecomposable projective (and injective) $A$-module whose radical is
$S$. Moreover, the Loewy length of $A$ must be two and $A$ must be a
Nakayama algebra. The non-projective indecomposable $A$-modules are all simple,
and each simple $A$-module is a node satisfying $\D \Tr (S) \cong \Omega_A(S)$
and $\Omega^m_A(S) \cong S$ for some positive integer $m$. This justifies
the following notation:

\begin{defn}
Let $A$ be a self-injective algebra with nodes.
Then $\rho(A)$ denotes the smallest positive integer $m$ such that
there exists a node $_AS$ and an isomorphism $S\cong \Omega^m_A(S)$.
\end{defn}

We will use various correspondences between
objects in two equivalent stable categories and in related categories.

Let $F: A\stmod\lra B\stmod$ be an equivalence of additive categories.
It induces
quasi-inverse equivalences of abelian categories:
$$\alpha: (A\stmod)\lmod \lraf{\cong} (B\stmod)\lmod\;\;\mbox{and}\;\;\beta: (B\stmod)\lmod \lraf{\cong}(A\stmod)\lmod.$$
Moreover, there are one-to-one correspondences
 $$
 \tilde{\alpha}:  A\lmod_{\mathscr{P}}\longleftrightarrow B\lmod_{\mathscr{P}}: \tilde{\beta}\quad \quad\mbox{and}\quad\quad
 \alpha': A\lmod_{\mathscr{I}}\longleftrightarrow B\lmod_{\mathscr{I}}:\beta'
 $$
such that
 \begin{align*}
   &\alpha(\underline{\Hom}_A(-,X))\cong \underline{\Hom}_B(-, \tilde{\alpha}(X))\quad \quad\mbox{and}\quad\quad
   \alpha(\Ext^1_A(-,Y)) \cong \Ext^1_B(-,\alpha'(Y)), \\
   &\beta(\underline{\Hom}_B(-,M))\cong \underline{\Hom}_A(-, \tilde{\beta}(M))\quad \quad\mbox{and}\quad\quad
   \beta(\Ext^1_B(-,N))\cong \Ext^1_A(-,\beta'(N))
 \end{align*}
where  $X\in A\lmod_{\mathscr{P}}, Y\in A\lmod_{\mathscr{I}}, M\in
 B\lmod_{\mathscr{P}}$ and $N\in B\lmod_{\mathscr{I}}$. When applied to
$A$-modules, $F$ and $\tilde{\alpha}$ coincide.

 Formally, we set $\tilde{\alpha}(P)=0$ and $\alpha'(I)=0$ when
 $_AP$ is projective and $_AI$ is injective. Readers familiar with the
 Auslander Reiten translation $\D\!\Tr$, may observe that the correpondences
 are related by $\mathrm{DTr}\tilde{\alpha}\cong \alpha'\mathrm{DTr}$.

 When working with generators, the following correspondences will be used:
 \begin{align*}
 \Phi\!\!: & A\lmod \to B\lmod\qquad U\mapsto \tilde{\alpha}(U)\oplus\bigoplus_{Q\in\mathscr{P}_B}Q,\\
 \Psi\!\!: & B\lmod\to A\lmod\qquad  V\mapsto \tilde{\beta}(V)\oplus\bigoplus_{P\in\mathscr{P}_A}P.
 \end{align*}

 Moreover, if $X$ is indecomposable, not projective, not injective
 and not a node, then $\tilde{\alpha}(X)\cong\alpha'(X)$ (see
 \cite[Lemma 3.4]{AR78}).
\medskip

A particular set of nodes will turn out to play a
crucial role (compare \cite[Section 3]{G04}):

\begin{defn}
Let $\mathfrak{n}_F(A)$ be the set of isomorphism classes
of indecomposable, neither projective nor injective $A$-modules $U$
 such that $\tilde{\alpha}(U)\ncong \alpha'(U)$.
\end{defn}

By definition, $\mathfrak{n}_F(A)$ depends on the given equivalence $F$. In
this situation, $\mathfrak{n}_{F^{-1}}(B)$ denotes the set of isomorphism classes
of indecomposable, neither projective nor injective $B$-modules $V$
such that $\tilde{\beta}(V)\ncong \beta'(V)$.

The elements of $\mathfrak{n}_F(A)$ are nodes; in
general, $\mathfrak{n}_F(A)$ is a proper subset of the set of isomorphism
classes of nodes. It can be empty although there are nodes.

 \subsection{Invariance under general stable equivalences and
under triangulated stable equivalences}
 We state the main result and briefly outline its proof, which will occupy the
 rest of this section. We also provide examples showing that
 the assumptions are necessary and algebras with nodes really behave
 differently.

 \begin{thm}\label{thm:rigidity dimension is invariant under stable equivalence}
   (I) Let $A$ and $B$ be stably equivalent algebras.
   Suppose $A$ has no nodes. Then $\cf(A)\leq \cf(B)$. In particular,
   if neither $A$ nor $B$ has nodes, then $\cf(A)=\cf(B)$. \smallskip \\
(II) More precisely, let $A$ and $B$ be stably equivalent by an equivalence $F$.
   If $\mathfrak{n}_F(A)$ is empty, then $\cf(A)\leq \cf(B)$. In particular,
   if both $\mathfrak{n}_F(A)$ and $\mathfrak{n}_{F^{-1}}(B)$ are empty, then $\cf(A)=\cf(B)$.
 \end{thm}

Statement (II) implies (I), since $\mathfrak{n}_F(A)$ is contained in the set
of nodes. Statement (I) may, however, be easier to apply, since it uses a
property of the algebra, not of the given equivalence. By (I), the rigidity
dimension of an algebra without nodes is minimal in its
stable equivalence class.
\smallskip

The stable module category of a self-injective algebra carries additional
structure; it is a triangulated category (see \cite{Happel,ARS95}). A
stable equivalence between self-injective algebras may be a triangulated
equivalence or just an additive equivalence.
It turns out that triangulated
equivalences do respect rigidity dimension even when there are nodes. Also,
in the additive case, more can be said when the algebras are self-injective.

\begin{thm}\label{Stable equivalences of self-injective algebras}
   Let $A$ and $B$ be stably equivalent self-injective algebras. Then:
\\
   (I)  $\cf(A)=\cf(B)$ in each of
   the following three cases:

    \quad $(1)$ $A$ has no nodes.

    \quad $(2)$ $A$ and $B$ are symmetric algebras.

    \quad $(3)$ $A\stmod$ and $B\stmod$ are equivalent as triangulated categories.
    \smallskip
\\
 (II) If both $A$ and $B$ have nodes, then
$$|\cf(A)-\cf(B)|\leq |\rho(A)-\rho(B)|.$$
In particular,
$\cf(A)<\infty $ if and only if $\cf(B)<\infty$.
\end{thm}

 In the proof of the two theorems,
we will have to control what happens to generator-cogenerators
 and their endomorphism algebras, under stable equivalences. In every step,
 nodes will cause problems. Therefore, in the first part of the proof, we will
 have to use various correspondences between objects in the stable
 categories and also in the functor categories. These correspondences are
 compatible with Guo's results \cite{G04}, which thus can be used to compare
 global dimensions of endomorphism rings. The core of the proof then is to
 compare also dominant dimensions, which needs another technically involved
 subsection. Finally, everything can be put together to derive Theorem
 \ref{thm:rigidity dimension is invariant under stable equivalence}. To prove
Theorem \ref{Stable equivalences of self-injective algebras} we will in
addition use the description of self-injective algebras with nodes in
\cite{ARS95}.
\medskip

In general, that is in the presence of nodes,
   stable equivalences do not preserve rigidity dimensions,
 as the following examples illustrate.

 \begin{example}
(a)  Let $B$ be the $k$-algebra given by the quiver
 $\xymatrix{1 \ar@/^/[r]^\alpha & 2 \ar@/^/[l]^\beta}$
 with relations $\{\alpha\beta, \beta\alpha\}$. Then $B$ is self-injective with
 radical square zero.
 $B$ is stably equivalent to $A:=T_2(k)\times T_2(k)$, where
 $T_2(k):=\begin{pmatrix}
              k  &0\\
              k  & k
 \end{pmatrix}$.
 The algebra $B$ has two nodes, but $A$ has no nodes. The rigidity dimensions of
 $B$ and $A$ can be calculated as follows.

 Since $T_2(k)$ is hereditary,  $\cf(T_2(k))=2$ by Theorem \ref{injective}.
 Together with Proposition \ref{prop:rigidity dimension blockwise}, $\cf(A)=\cf(T_2(k))=2$.
 To calculate $\cf(B)$, let $S_1$ and $S_2$ denote the simple $B$-modules corresponding
 to the vertices $1$ and $2$, respectively. Note that $B$ has only four basic
generator-cogenerators (up to isomorphism):  $B$, $B\oplus S_1, B\oplus S_2$
and $B\oplus S_1\oplus S_2$.
 Moreover the endomorphism algebras, except for $B$ itself,
have finite global dimension;
 $\domdim\End_B(B\oplus S_1)= 3=\domdim\End_B(B\oplus S_2)$, but $\domdim\End_B(B\oplus S_1\oplus S_2) = 2$. Thus $\cf(B)=3>\cf(A)$, illustrating that the
 inequality in Theorem \ref{thm:rigidity dimension is invariant under stable equivalence} cannot be improved. In this example, the set
$\mathcal{GCN}_{F^{-1}}(B)$, to be introduced in the next subsection,
consists only of the Auslander generator
 $B\oplus S_1\oplus S_2$ since $B$ has two nodes $S_1$ and $S_2$.

 (b) Now, we give an example showing that two stably equivalent self-injective algebras may have different rigidity dimensions. Let $B$ be as in (a) and let
 $C:=k[x]/(x^2)\times k[x]/(x^2)$. This is a self-injective algebra.
 To check that $B$ and $C$ are stably equivalent it is enough to count
 indecomposable objects in $B\stmod$ and in $C\stmod$.
 But $C$ has nodes and $\cf(C)=2\neq \cf(B)$. As the stable category of a Frobenius category is triangulated, both $B\stmod$ and $C\stmod$ are triangulated categories. However, the shift functor of $B\stmod$ permutes simple modules, while the shift functor of $C\stmod$ is the identity. This means that $B\stmod$ and $C\stmod$ are equivalent as additive categories, but not as triangulated categories.
\end{example}

\subsection{Further preparations for the proofs}
Throughout this subsection, $A$ and $B$ are stably equivalent algebras,
possibly with nodes. We continue setting up correspondences between
objects in the two stable categories and in related categories.
Let $F: A\stmod\lra B\stmod$ be an equivalence of additive categories.

 Denote by $\mathscr{P}_A$ and $\mathscr{I}_A$ the set of isomorphism
 classes of indecomposable projective and injective $A$-modules, respectively.
  Further, set
 $$
 \bigtriangleup_A:=\mathfrak{n}_F(A)\dot{\cup}(\mathscr{P}_A\setminus\mathscr{I}_A)
 \quad\mbox{and}\quad \bigtriangledown_A:=\mathfrak{n}_F(A)\dot{\cup}(\mathscr{I}_A\setminus\mathscr{P}_A).
 $$
Let $\bigtriangleup_A^c$ be the class of indecomposable, non-injective
 $A$-modules which do not belong to $\bigtriangleup_A$.
Then each module $Y\in A\lmod_{\mathscr{I}}$ admits a unique decomposition (up to isomorphism)
 $$Y\cong Y_\triangle\oplus Y'$$ with
 $Y_\triangle\in\add(\bigtriangleup_A)$ and $Y'\in\add(\bigtriangleup_A^c)$.
 The module $Y_\triangle$ is called the \emph{$\bigtriangleup_A$-component} of $Y$.

 In the following, we denote by $\mathcal{GCN}_F(A)$
the class of basic $A$-modules $X$
 which are generator-cogenerators with $\mathfrak{n}_F(A)\subseteq \add(X)$.
 In particular, if $A$ has no nodes, then $\mathcal{GCN}_F(A)$ is exactly
 the class of basic generator-cogenerators for $A\lmod$. In the same
situation, we similarly use the notation $\mathcal{GCN}_{F^{-1}}(B)$.

Here are some results from \cite{G04} for later use.

 \begin{lem}\textrm{\bf(compare \cite[Lemmas 3.1 and 3.2]{G04})} \label{lem:refinement on bijection in stable equivalences}

 $(1)$ There are one-to-one correspondences
 $$
 \tilde{\alpha}: \bigtriangledown_A\longleftrightarrow \bigtriangledown_B:\tilde{\beta},
 \quad \alpha':  \bigtriangleup_A\longleftrightarrow \bigtriangleup_B: \beta'
 \quad\mbox{and}\quad
 \alpha':\bigtriangleup_A^c\longleftrightarrow \bigtriangleup_B^c: \beta'.
 $$

$(2)$ The correspondences $\Phi$ and $\Psi$
 restrict to one-to-one correspondences between
 $\mathcal{GCN}_F(A)$ and $ \mathcal{GCN}_{F^{-1}}(B)$.
 Moreover, if $X\in \mathcal{GCN}_F(A)$, then $\Phi(X)\cong \alpha'(X)
 \oplus \bigoplus_{I\in\mathscr{I}_B}I$.

 \end{lem}
 \begin{proof}
 $(1)$ By \cite[Lemma 3.1]{G04}, we only need to show that
 $X\in \mathfrak{n}_F(A)$ implies $\alpha'(X)\in\bigtriangleup_B$.

 Let $X\in \mathfrak{n}_F(A)$. Then $X$ is neither projective nor injective.
 Consequently,  $\alpha'(X)$ is not injective.
 Recall that
 $\tilde{\beta}(V)\cong \beta'(V)$ for any $V\in\bigtriangleup_B^c$.
 Suppose $\alpha'(X)\in\bigtriangleup_B^c$.
Then $\alpha'(X)$  is not projective and
 $\tilde{\beta}(\alpha'(X))\cong \beta'(\alpha'(X))\cong X$. It follows that  $\alpha'(X)\cong\tilde{\alpha}\tilde{\beta}(\alpha'(X))\cong\tilde{\alpha}(X)$.
 This is a contradiction since $X\in \mathfrak{n}_F(A)$. Thus $\alpha '(X)\in\bigtriangleup_B$.

Note that $\tilde{\alpha}(U)\cong\alpha'(U)$ for any $U\in \bigtriangleup_A^c$. Now, $(2)$ follows from $(1)$.
 \end{proof}


 Generator-cogenerators in $\mathcal{GCN}_F(A)$ are better to control under
stable  equivalences. In particular, global dimensions of their
 endomorphism algebras are preserved under stable equivalences:

 \begin{lem}\textrm{\bf (\cite[Lemma 3.5]{G04})}\label{cor:rep dim is invariant under stable equivalences}
 If $X\in\mathcal{GCN}_F(A)$, then
 $$\gdim\End_A(X)=\gdim \End_B(\Phi(X)).$$
 \end{lem}

\subsection{Stable equivalences and dominant dimension
of endomorphism rings}

To compare rigidity dimension under stable equivalences, we need an analogue of
Lemma \ref{cor:rep dim is invariant under stable equivalences} for dominant
dimensions; this is the main part of the proof of Theorem
\ref{thm:rigidity dimension is invariant under stable equivalence}.

 \begin{prop}\label{prop:g5} Let $A$ and $B$ be stably equivalent by an
equivalence $F$.
   If $X\in\mathcal{GNC}(A)$, then
   $$\domdim \End_A(X)=\domdim \End_B(\Phi(X)).$$
 \end{prop}

 The proof of Proposition \ref{prop:g5} will need three lemmas. The first one
 extends the first part of \cite[Lemma 3.3]{G04}.

 \begin{lem} \label{lem:g3}
   Assume that ${_A}Z$ is indecomposable and non-projective. Let
   $$0\lra X\oplus X'\lra Y\oplus P \lraf{g}  Z\lra 0$$
   be an exact sequence of $A$-modules without a split exact sequence as a
   direct summand,
   such that ${X\in\add(\bigtriangleup_A^c)}$, $X'\in\add(\bigtriangleup_A)$,
   $Y\in A\lmod_{\mathscr{P}}$ and $P\in\add({_A}A)$.
   Then there is an exact sequence of $B$-modules
   $$
   0\lra \tilde{\alpha}(X)\oplus \alpha'(X')\lra \tilde{\alpha}(Y)\oplus Q\lraf{g'} \tilde{\alpha}(Z)\lra 0
   $$
   without a split direct summand such that $g'=F(g)$ in $B\stmod$ with $Q\in\add(B)$.
 \end{lem}

 \begin{proof}
 By the proof of the first part of \cite[Lemma 3.3]{G04}, the following two statements
 hold.
 \begin{itemize}
   \item[\rm(1)]
   There is an exact sequence of $B$-modules
   $$
   0\lra N\oplus N'\lra \tilde{\alpha}(Y)\oplus Q\lraf{g'} \tilde{\alpha}(Z)\lra 0
   $$
   without a split direct summand such that $g'=\alpha(g)$ in
   $B\stmod$, and that $N\in\add(\bigtriangleup_B^c)$,
   ${N'\in\add(\bigtriangleup_B)}$ and $Q\in\add(B)$.
   \item[\rm(2)] $\alpha\big(\Ext^1_A(-,X\oplus X')\big)
   \cong \Ext^1_B(-,N\oplus N')$ in $(B\stmod)\lmod$.
 \end{itemize}
 Since $X\oplus X'\in A\lmod_{\mathscr{I}}$, there is an isomorphism
 $\alpha\big(\Ext^1_A(-,X\oplus X')\big)\cong
 \Ext^1_A(-,\alpha'(X)\oplus \alpha'(X'))$.
 Note that $\alpha'(X)\cong\tilde{\alpha}(X)$ since $X\in\add(\bigtriangleup_A^c)$.
 Thus $$\Ext^1_B(-,N\oplus N')\cong \alpha\big(\Ext^1_A(-,X\oplus X')\big)\cong
 \Ext^1_B(-,\tilde{\alpha}(X)\oplus \alpha'(X')).$$
 This implies $\tilde{\alpha}(X)\oplus \alpha'(X')\cong N\oplus N'$
 in $B\lmod_{\mathscr{I}}$. Since $\tilde{\alpha}(X)\in\add(\bigtriangleup_B^c)$ and
 $\alpha'(X')\in \bigtriangleup_B$ by Lemma \ref{lem:refinement on bijection in stable equivalences}(1),
 there are isomorphisms $\tilde{\alpha}(X)\cong N$ and $\alpha'(X')\cong N'$.
 By $(1)$, Lemma \ref{lem:g3} follows.
 \end{proof}

 The second lemma establishes a connection between different syzygy modules under stable equivalences.

 \begin{lem}\label{lem:g4}
   Let $X\in A\lmod$ and $n$ a positive integer. Then
   $$
   \tilde{\alpha}\big(\Omega_A^n(X)\big)\oplus \bigoplus_{j=1}^n\Omega_B^{n-j}\big(\alpha'(\Omega_A^j(X)_\triangle)\big)
   \cong \Omega_B^n(\tilde{\alpha}(X))\oplus\bigoplus_{j=1}^n\Omega_B^{n-j}\big(\tilde{\alpha}(\Omega_A^j(X)_\triangle)\big),
   $$
   where $\Omega_A^j(X)_\triangle$ stands for the $\bigtriangleup_A$-component of the $A$-module $\Omega_A^j(X)$.
 \end{lem}

 \begin{proof} Lemma \ref{lem:g4} will be shown by induction on $n$.
 If $n=1$, then it suffices to check
 $$(\ast)\quad
 \tilde{\alpha}\big(\Omega_A(X)\big)\oplus\alpha'(\Omega_A(X)_\triangle)
 \cong \Omega_B(\tilde{\alpha}(X))\oplus\tilde{\alpha}(\Omega_A(X)_\triangle).$$
 If $X$ is projective, then both sides are zero, and there is nothing to prove.
 It is enough to show $(\ast)$ when $X$ is indecomposable and non-projective.
 Let $$0\lra \Omega_A(X)\lra P\lra X\lra 0$$ be an exact sequence of $A$-modules such that
 $P$ is a projective cover of $X$. Then this sequence contains no split direct summand and
 $Y:=\Omega_A(X)\in A\lmod_{\mathscr{I}}$. So $Y$ has a decomposition
 as $Y\cong Y_\triangle\oplus Z$, where $Y_\triangle\in\add(\bigtriangleup_A)$ and
 $Z\in\add(\bigtriangleup_A^c)$.
 Applying Lemma \ref{lem:g3} to the sequence
 $$0\lra Z\oplus Y_\triangle \lra P\lra X\lra 0$$
 yields the following exact sequence of $B$-modules
 $$
 0\lra \tilde{\alpha}(Z)\oplus\alpha'(Y_\triangle)\lra Q\lra\tilde{\alpha}(X)\lra 0
 $$
 without split direct summands, such that $Q\in\add(B)$. Thus $Q$ is a projective cover of $\tilde{\alpha}(X)$ and further,
 $\Omega_B(\tilde{\alpha}(X))\cong\tilde{\alpha}(Z)\oplus \alpha'(Y_\triangle)$.
 Hence $$\tilde{\alpha}(Y)\oplus\alpha'(Y_\triangle)
   \cong \tilde{\alpha}(Y_\triangle)\oplus \tilde{\alpha}(Z) \oplus \alpha'(Y_\triangle)
   \cong\Omega_B(\tilde{\alpha}(X))\oplus \tilde{\alpha}(Y_\triangle).$$
This shows the isomorphism $(\ast)$.

Let $n\geq 2$. Suppose that for any $A$-module $U$, there is an isomorphism of $B$-modules
 $$
 (\ast\ast)\quad\tilde{\alpha}\big(\Omega_A^{n-1}(U)\big)\oplus \bigoplus_{j=1}^{n-1}
 \Omega_B^{n-1-j}\big(\alpha'(\Omega_A^j(U)_\triangle)\big)
 \cong \Omega_B^{n-1}(\tilde{\alpha}(U))\oplus\bigoplus_{j=1}^{n-1}
 \Omega_B^{n-1-j}\big(\tilde{\alpha}(\Omega_A^j(U)_\triangle)\big).
 $$
Choosing $U=Y=\Omega_A(X)$ gives
 $$
 \tilde{\alpha}\big(\Omega_A^{n}(X)\big)\oplus \bigoplus_{j=1}^{n-1}
 \Omega_B^{n-1-j}\big(\alpha'(\Omega_A^{j+1}(X)_\triangle)\big)
 \cong \Omega_B^{n-1}(\tilde{\alpha}(\Omega_A(X)))\oplus\bigoplus_{j=1}^{n-1}
 \Omega_B^{n-1-j}\big(\tilde{\alpha}(\Omega_A^{j+1}(X)_\triangle)\big).
 $$
 Thus
 \begin{align*}
   &\tilde{\alpha}\big(\Omega_A^n(X)\big)\oplus \bigoplus_{j=1}^n\Omega_B^{n-j}\big(\alpha'(\Omega_A^j(X)_\triangle)\big)\\
   \cong \,\, &
   \tilde{\alpha}\big(\Omega_A^{n}(X)\big)\oplus \bigoplus_{j=1}^{n-1}
   \Omega_B^{n-1-j}\big(\alpha'(\Omega_A^{j+1}(X)_\triangle)\big)
   \oplus \Omega_B^{n-1}(\alpha'(\Omega_A(X)_\triangle)) \\
   \cong \, \, &
   \Omega_B^{n-1}(\tilde{\alpha}(\Omega_A(X))\oplus\alpha'(\Omega_A(X)_\triangle))
   \oplus\bigoplus_{j=1}^{n-1}
   \Omega_B^{n-1-j}\big(\tilde{\alpha}(\Omega_A^{j+1}(X)_\triangle)\big)\qquad \text{\big(by\;($\ast\ast$)\big)}\\
   \cong \, \, &
   \Omega_B^{n}(\tilde{\alpha}(X))\oplus
   \Omega_B^{n-1}\big(\tilde{\alpha}(\Omega_A(X)_\triangle)\big)
   \oplus\bigoplus_{j=1}^{n-1}\Omega_B^{n-(j+1)}\big(\tilde{\alpha}(\Omega_A^{j+1}(X)_\triangle)\big)\qquad \text{\big(by\;($\ast$)\big)}\\
   \cong \, \, & \Omega_B^n(\tilde{\alpha}(X))\oplus\bigoplus_{j=1}^n\Omega_B^{n-j}\big(\tilde{\alpha}(\Omega_A^j(X)_\triangle)\big).
 \end{align*}
\end{proof}

\medskip
The third lemma identifies extension groups of modules preserved under
stable equivalences.

 \begin{lem}\label{lem:extention}
   Let $X\in A\lmod$, $Y\in\mathcal{GCN}_F(A)$ and $n$ a positive integer.
   \begin{itemize}
    \item[\rm(1)]  Then $\Ext_A^1(X,Y)\cong \Ext_B^1(\Phi(X),\Phi(Y))$.
    \item[\rm(2)]  If $\Ext_B^i(N,\Phi(Y))=0$
    for each $N\in\bigtriangledown_B$ and $1\leq i\leq n$, then
    $$\Ext^{n+1}_A(X,Y)\cong \Ext_B^{n+1}(\Phi(X),\Phi(Y)).$$
   \end{itemize}
 \end{lem}
 \begin{proof}
 Set $\mathcal{C}:=(A\stmod)\lmod$ and $\mathcal{D}:=(B\stmod)\lmod$.
 Then $\alpha :\mathcal{C}\to \mathcal{D}$ is an equivalence
 with a quasi-inverse $\beta$. By Yoneda's Lemma,
 \[
 \Ext_A^1(X,Y)\cong \Hom_\mathcal{C}\big(\underline{\Hom}_A(-,X),\Ext_A^1(-,Y)\big)
 \cong\Hom_\mathcal{D}\big(\alpha(\underline{\Hom}_A(-,X)),\alpha(\Ext_A^1(-,Y))\big).
 \]
 Since $\alpha(\underline{\Hom}_A(-,X))\cong\underline{\Hom}_B(-,\tilde{\alpha}(X))$ and
 $\alpha(\Ext_A^1(-,Y))\cong\Ext_B^1(-,\alpha'(Y))$, it follows that
  \[
 \Ext_A^1(X,Y)\cong \Ext_B^1(\tilde{\alpha}(X),\;\alpha'(Y)).
 \]
 Moreover, by Lemma \ref{lem:refinement on bijection in stable equivalences}(2),
$\Phi(Y)=\alpha'(Y)\oplus \bigoplus_{I\in\mathscr{I}_B}I$, since
$Y\in \mathcal{GCN}_F(A)$.
Also, $\Phi(X)=\tilde{\alpha}(X)\oplus\bigoplus_{Q\in\mathscr{P}_B}Q$,
by definition. Consequently,
 $\Ext_A^1(X,Y)\cong\Ext_B^1(\Phi(X),\Phi(Y))$
 Thus $(1)$ holds

The proof of $(1)$ also implies
 \[
 \Ext_A^{n+1}(X,Y)\cong\Ext_A^1(\Omega_A^n(X),\,Y)
 \cong\Ext_B^1(\tilde{\alpha}(\Omega_A^n(X)),\,\alpha'(Y))
 \cong \Ext_B^1(\tilde{\alpha}(\Omega_A^n(X)),\,\Phi(Y)).
 \]
 By Lemma \ref{lem:g4},
 $\tilde{\alpha}\big(\Omega_A^n(X)\big)\oplus L\cong \Omega_B^n(\tilde{\alpha}(X))\oplus R$ in $B\lmod$, where
 $$
 L:=\bigoplus_{j=1}^n\Omega_B^{n-j}\big(\alpha'(\Omega_A^j(X)_\triangle)\big)
 \quad\mbox{and}\quad R:=\bigoplus_{j=1}^n\Omega_B^{n-j}\big(\tilde{\alpha}(\Omega_A^j(X)_\triangle)\big).
 $$
 Since $\Omega_A^j(X)_\triangle\in\add(\bigtriangleup_A)$ and
 $\bigtriangleup_A= \mathfrak{n}_F(A) \dot{\cup}(\mathscr{P}_A\setminus\mathscr{I}_A)$, it follows from
 Lemma \ref{lem:refinement on bijection in stable equivalences}(1) that
 \[
 \alpha'(\Omega_A^j(X)_\triangle)\in\add(\bigtriangleup_B)
 \quad\mbox{and}\quad \tilde{\alpha}(\Omega_A^j(X)_\triangle)\in\add\big(\tilde{\alpha}(\mathfrak{n}_F(A))\big)\subseteq\add(\bigtriangledown_B)
 \]
 where $\bigtriangledown_B:=\mathfrak{n}_{F^{-1}}(B)\dot{\cup}(\mathscr{I}_A\setminus\mathscr{P}_A)$.
 So, if $\Ext_B^i(N,\Phi(Y))=0$ for each $N\in\bigtriangledown_B$
 and $1\leq i\leq n$, then
 $\Ext_B^1(L,\Phi(Y))=0=\Ext_B^1(R,\Phi(Y))$, and thus
 \begin{align*}
 \Ext_B^1(\tilde{\alpha}(\Omega_A^n(X)),\,\Phi(Y))& \cong
 \Ext_B^1(\Omega_B^n(\tilde{\alpha}(X)),\,\Phi(Y))\\ &\cong
 \Ext_B^{n+1}(\tilde{\alpha}(X),\,\Phi(Y))\cong\Ext_B^{n+1}(\Phi(X),\,\Phi(Y)).
 \end{align*}
 This shows $(2)$.
 \end{proof}

 \medskip

{\bf Proof of Proposition \ref{prop:g5}}.

   If \mbox{$\domdim \End_B(\Phi(X)) = n+2$} for some $n\in\mathbb{N}$,
then $\Ext^i_B(\Phi(X),\Phi(X))=0$ for $1\leq i\leq n$ and
   $\Ext^{n+1}_B(\Phi(X),\Phi(X))\neq 0$ due to
   Theorem \ref{thm:Muller characterization}.
By Lemma \ref{lem:refinement on bijection in stable equivalences}(2),
$\Phi(X)\in\mathcal{GCN}_{F^{-1}}(B)$. Thus,
${\bigtriangledown_B\subseteq\add(\Phi(X))}$.
   Therefore, Lemma \ref{lem:extention} implies
   $\Ext^j_A(X,X)\cong \Ext^j_B(\Phi(X),\Phi(X))$
   for \mbox{$1\leq j \leq  n+1$.} Then
  $\domdim\End_A(X)=n+2$, again by Theorem \ref{thm:Muller characterization}.
   Similarly, if $\domdim \End_B(\Phi(X))=\infty$, then $\domdim \End_A(X)=\infty$.
   $\square$

\medskip
\noindent \textbf{Remark.} When neither $A$ nor $B$ has nodes, both Lemma \ref{lem:g4} and Lemma \ref{lem:extention} can be simplified.
In Lemma \ref{lem:g4}, the isomorphism becomes $\tilde{\alpha}\big(\Omega_A^n(X)\big)\oplus Q\cong \Omega_B^n(\tilde{\alpha}(X))$,
where $Q$ is a projective $B$-module without injective direct summands. This implies a stronger form of Lemma \ref{lem:extention}:
If $X\in A\lmod$ and $Y\in\mathcal{GCN}_F(A)$, then $\Ext^{n}_A(X,Y)\cong \Ext_B^{n}(\Phi(X),\Phi(Y))$ for any $n\geq 1$.
Note that both isomorphisms can be obtained from \cite[Theorem 3.6]{AR78} and \cite[Section 1, Corollary, and Proposition 2.2]{MV90}.
Thus, under the stronger assumption that $A$ and $B$ are two stably equivalent algebras without nodes and without semi-simple blocks, Proposition \ref{prop:g5}
can also be derived from the results in \cite{MV90}.

\subsection{Completion of proofs, and an application to representation dimension}

{\bf Proof of Theorem \ref{thm:rigidity dimension is invariant under stable equivalence}}.

 Let $X\in\mathcal{GCN}_F(A)$. Then  $\Phi(X)\in\mathcal{GCN}_{F^{-1}}(B)$.
 By Lemma \ref{cor:rep dim is invariant under stable equivalences}, both
$\End_A(X)$ and $\End_B(\Phi(X))$ have the
 same global dimension. Moreover, by Proposition \ref{prop:g5}, they have the
same dominant dimension.
 When $\mathfrak{n}_F(A)$ is the empty set,
then $\mathcal{GCN}_F(A)$ is exactly the
class of basic generator-cogenerators in $A\lmod$. It follows that
$\cf(A)\leq \cf(B)$. If $\mathfrak{n}_{F^{-1}}(B)$ is also empty,
then $\cf(B)\leq \cf(A)$, and thus
 $\cf(A)=\cf(B)$. $\square$

\medskip

\medskip
To prepare for the proof of
Theorem \ref{Stable equivalences of self-injective algebras},
we describe the rigidity dimensions of Nakayama self-injective algebras with
radical square zero.

\begin{prop}\label{Nakayama self-injective algebra with radical square zero}
Suppose that $A$ is an indecomposable non-simple Nakayama self-injective algebra with
radical square zero. Let $e$ be the number of isomorphism classes of simple $A$-modules.
Then  $\cf(A) = e+1$.
 \end{prop}
\begin{proof}
   Since the radical square of $A$ is zero, every indecomposable, non-projective $A$-module
   is simple. So, for any generator $M$ in $A\lmod$, if $\gdim\End_A(M)<\infty$, then $M$ contains at least one simple module, say $S$, as a direct summand. In particular,
   by Theorem \ref{thm:Muller characterization}
   \[
    \domdim\End_A(M)\leq \domdim \End_A(A\oplus S)
   \]
Note that $\Omega_A^{e}(S)\cong S$ and $\{\Omega^i_A(S)\mid 1\leq i\leq e\}$ is the complete set of isomorphism classes
of indecomposable, non-projective $A$-modules. The following equalities
$$ S^{\bot_{(e-1)}}=\add(A\oplus S)=^{\bot_{(e-1)}}\!S. $$
can be verified by writing down projective and injective resolutions of $S$.
This implies that $A\oplus S$ is a maximal $(e-1)$-orthogonal $A$-module. By
\cite[Theorem 0.2]{Iyama07},
$\domdim\End_A(A\oplus S)=e+1=\gdim\End_A(A\oplus S).$
Thus $\cf(A)=\domdim \End_A(A\oplus S)=e+1$ as claimed.
  \end{proof}

{\bf Proof of Theorem \ref{Stable equivalences of self-injective algebras}}

Since $A$ and $B$ are self-injective, it follows from \cite[X.1: Proposition 1.6]
{ARS95} that $\tilde{\alpha}$ and $\tilde{\beta}$ restrict to one-to-one correspondences between the set of isomorphism classes of nodes of $A$ and
the set of nodes of $B$. If $A$ has no nodes, then so does $B$. Thus $(1)$ holds by Theorem \ref{thm:rigidity dimension is invariant under stable equivalence}.

Suppose that both $A$ and $B$ have nodes and no semisimple blocks. We first show (II), and then use (II) to show $(2)$ and $(3)$.

Let $A=A_1\times A_2$ and $B=B_1\times B_2$ be decompositions of algebras,
such that $A_2$ and $B_2$ are the products of all blocks of $A$ and $B$ without nodes,
respectively. In other words,
all nodes of $A$ and $B$ only belong to $A_1\lmod$ and $B_1\lmod$, respectively. Then, all indecomposable non-projective $A_1$-modules
(and similarly, $B_1$-modules) are nodes. This is a consequence of the
following facts:

As already mentioned, by \cite[X.1: Proposition 1.8]{ARS95}, if $\Lambda $ is
an indecomposable self-injective algebra with nodes, then it is a
Nakayama algebra
with radical square zero. In this case, each simple $\Lambda$-module $N$ is a node, $DTr(N)\cong \Omega_\Lambda(N)$ and
$\Omega_A^{\rho(\Lambda)}(N)\cong N$. Moreover, $\{\Omega^i_\Lambda(N)\mid 1\leq i\leq \rho(\Lambda)\}$ is the complete set of isomorphism classes
of indecomposable objects in $\Lambda\stmod$.  By Proposition
\ref{Nakayama self-injective algebra with radical square zero},
$\cf(\Lambda)=\rho(\Lambda)+1$.

Consequently, both $A_1$ and $B_1$ are
products of indecomposable Nakayama algebras with radical square zero.
Moreover,  $\tilde{\alpha}$ and $\tilde{\beta}$ induce a stable equivalence
between $A_1$ and $B_1$ and a stable equivalence between $A_2$ and $B_2$.
By Proposition \ref{prop:rigidity dimension blockwise}, $\cf(A)=\min\{\rho(A_1)+1, \cf(A_2)\}$ and $\cf(B)=\min\{\rho(B_1)+1, \cf(B_2)\}$.
Since $A_2$ and $B_2$ have no nodes, $\rho(A_1)=\rho(A)$ and $\rho(B_1)=\rho(B)$. By $(1)$, $\cf(A_2)=\cf(B_2)$.
Now, it is easy to check the inequality $|\cf(A)-\cf(B)|\leq |\rho(A)-\rho(B)|.$ This shows (II).

In particular, (II) implies that if $\rho(A)=\rho(B)$, then $\cf(A)=\cf(B)$. A sufficient condition to guarantee $\rho(A)=\rho(B)$ is that
$\tilde{\alpha}(\Omega_A(S))\cong\Omega_B(\tilde{\alpha}(S))$ in $B\lmod_{\mathscr{P}}$ for any node $S$ of $A$. In this case,
both $S$ and $\tilde{\alpha}(S)$ are $\Omega$-periodic of the same period.

When $A$ and $B$ are symmetric algebras, it follows from \cite[X.1: Proposition 1.12]{ARS95} that the correspondence $\tilde{\alpha}$ between objects in  $A\lmod_{\mathscr{P}}$ and $B\lmod_{\mathscr{P}}$ commutes with the syzygy functor $\Omega$. This shows $(2)$.
Recall that the shift functor of the triangulated category $A\stmod$
is the cosyzygy functor $\Omega_A^{-1}$. So, in $(3)$, $\tilde{\alpha}$
commutes with $\Omega^{-1}$, and thus also with $\Omega$. This shows $(3)$.
$\square$

\medskip

\textbf{Remark.}
In Theorem \ref{Stable equivalences of self-injective algebras} $(2)$ and
$(3)$, both $\mathfrak{n}_F(A)$ and $\mathfrak{n}_{F^{-1}}(B)$ are empty.
This follows from $(\ast)$ in the proof of Lemma \ref{lem:g4}, or alternatively
since the functors $\D \Tr$ and $\Omega$ coincide when applied to
nodes, and since the correspondences $\tilde{\alpha}$ and $\tilde{\beta}$
commute with $\Omega$ as is shown in the proof of Theorem
\ref{Stable equivalences of self-injective algebras}.

\medskip

Finally, we explain how our results can be used to compare higher representation dimensions of stably equivalent algebras.
Recall that classical representation dimension $\rep_1$  is preserved under arbitrary stable equivalences of algebras (see \cite{G04}).
However, the following result suggests that in the case of higher representation
dimension, the situation is similar to rigidity dimension and depending on the
presence of nodes.

 \begin{cor}\label{higher representation dimension}
   Let $A$ and $B$ be stably equivalent by an equivalence $F$,
 and let $n \in {\mathbb N}$.
   Suppose that $\mathfrak{n}_F(A)$ is empty and $n+1\leq\cf(A)$.
Then $\rep_n(B)\leq \rep_n(A)<\infty$.
If additionally $\mathfrak{n}_{F^{-1}}(B)$ is empty, too, then
$\rep_n(A)=\rep_n(B)$.
 \end{cor}

 \begin{proof}
   Since $ n+1\leq \cf(A)$, Proposition
\ref{prop:relation between cf dim and rep dim} implies $\rep_n(A)<\infty$.
The inequality $\cf(A)\leq \cf(B)$ follows from Theorem
\ref{thm:rigidity dimension is invariant under stable equivalence}.
This implies $n+1\leq \cf(B)$, and thus $\rep_n(B)<\infty$. By assumption,
the set $\mathfrak{n}_F(A)$ is empty and $\mathcal{GCN}_F(A)$
is exactly the class of basic generator-cogenerators in $A\lmod$. Now,
Corollary \ref{higher representation dimension}
 follows from Proposition \ref{prop:g5} and Lemma
\ref{cor:rep dim is invariant under stable equivalences}.
 \end{proof}

 \section{Stable equivalences and invariance II}
 \label{sec:stableqMorita}

In this section, we show invariance of rigidity dimension under stable equivalences
of adjoint type, which implies invariance under stable equivalences of
Morita type under very mild assumptions, and thus also invariance
under certain derived equivalences.

\subsection{Definitions and main result}

\begin{defn}
   Two algebras $A$ and $B$ are \emph{stably
   equivalent of Morita type} if there exist
   an $(A,B)$-bimodule $M$ and a $(B,A)$-bimodule $N$ such that
   \begin{itemize}
     \item[\rm(i)] $M$ and $N$ are both projective as one sided modules,
     \item[\rm(ii)] $M\otimes_B N\cong A\oplus P$ as $A$-$A$-bimodules for
     some projective $A$-$A$--bimodule $P$,
     \item[\rm(iii)] $N\otimes_A M\cong B\oplus Q$
     as $B$-$B$-bimodules for some projective $B$-$B$-bimodule $Q$.
   \end{itemize}
   Further, if $(M\otimes_B-, N\otimes_A-)$ and $(N\otimes_A-, M\otimes_B-)$
   are adjoint pairs of functors, then $A$ and $B$ are \emph{stably equivalent of adjoint
   type}.
 \end{defn}
If $A$ and $B$ are stably equivalent of Morita type, then $(M\otimes_B-, N\otimes_A-)$ induce a stable equivalence between $A$ and $B$.

 \begin{thm}\label{thm:rigidity dimension is invariant under stable equivalences of adjoint type}
   (a) Let $A$ and $B$ be stably equivalent of adjoint type. Then
   $\cf(A)=\cf(B)$. \\
   (b) Let $A$ and $B$ be stably equivalent of Morita type. Then
   $\cf(A)=\cf(B)$ in each of the following three cases:
\\
$(1)$ $A$ and $B$ have no simple blocks.
\\
$(2)$ $A$ and $B$ are algebras over a perfect field $k$.
\\
$(3)$ $A$ and $B$ are self-injective algebras.
 \end{thm}

\noindent \textbf{Remark.} We don't know whether Theorem \ref{thm:rigidity dimension is invariant under stable equivalences of adjoint type} holds for general stable equivalences of Morita type.
Our proof of Theorem \ref{thm:rigidity dimension is invariant under stable equivalences of adjoint type} (b) depends on the relevant functors forming an
adjoint pair, and thus part (a) being applicable.
For an arbitrary stable equivalence of Morita type, it is not clear if tensor functors induced by two bimodules preserve injective modules.

\subsection{Proof of the main result}

In the proof of Theorem \ref{thm:rigidity dimension is invariant under stable equivalences of adjoint type}(b),
the following result will be used, which is likely to be known to experts.
 We thank Yuming Liu for pointing out the present proof.
 \begin{lem}\label{separable}
Let $A=A_1\times A_2$ and $B=B_1\times B_2$, where $A_1$ and $B_1$ are separable, and  $A_2$ and $B_2$ have no separable blocks. If $A$ and $B$ are stably equivalent of Morita type, then
so are $A_2$ and $B_2$.
\end{lem}

\begin{proof}
  It suffices to verify the following claim:

  If $\Lambda$ is a non-zero algebra and $S$ is a separable algebra, then  $\Lambda$ and $\Lambda\times S$ are stably equivalent of Morita type.

  For checking this, let $\Gamma:=\Lambda\times S$, $M:=\Lambda\times (\Lambda\otimes_kS)$ and $N:=\Lambda\times (S\otimes_k\Lambda)$. Then $M$ can be endowed with a $\Lambda$-$\Gamma$-bimodule structure:  For $\lambda, \lambda_1,\lambda_2\in\Lambda$ and $s,s'\in S$,
  $$\lambda(\lambda_1,\lambda_2\otimes s)=(\lambda\lambda_1,\,\lambda\lambda_2\otimes s)\;\;\mbox{and}\;\;
  (\lambda_1,\lambda_2\otimes s)(\lambda',s')=(\lambda_1\lambda',\lambda_2\otimes ss').$$
  Similarly, $N$ can be endowed with  a $\Gamma$-$\Lambda$-bimodule structure. Moreover,
  $M\otimes_\Gamma N\cong \Lambda\oplus (\Lambda\otimes_kS\otimes_k\Lambda)$ as $\Lambda$-$\Lambda$-bimodules and $N\otimes_\Lambda M\cong \Lambda\oplus \Lambda\otimes_kS\oplus S\otimes_k \Lambda\oplus S\otimes_k\Lambda\otimes_kS$ as $\Gamma$-$\Gamma$ bimodules. Then $\Lambda\otimes_kS\otimes_k\Lambda$ is a projective
  $\Lambda$-$\Lambda$-bimodule. As $$\Gamma\otimes_k\Gamma^{\rm{op}}\cong (\Lambda\otimes_k\Lambda^{\rm{op}})\times (\Lambda\otimes_k S^{\rm{op}}) \times (S\otimes_k\Lambda^{\rm{op}})\times (S\otimes_kS^{\rm{op}}),$$
$S\otimes_kS$, $\Lambda\otimes_kS$ and $ S\otimes_k \Lambda$ are projective
$\Gamma$-$\Gamma$-bimodules. The $S$-$S$-bimodule $S$ is projective, since
$S$ is a separable algebra. Furthermore, the multiplication map
$S\otimes_kS\to S$ is a homomorphism of $S$-$S$-bimodules, and $S$ is a direct summand of $S\otimes_kS$ as bimodules.
Since $S\otimes_k\Lambda\otimes_kS\cong (S\otimes_kS)^n$ with $n:=\dim(_k\Lambda)\geq 1$, it follows that $S$ is a direct summand of the projective bimodule $S\otimes_k\Lambda\otimes_kS$. Consequently, there is a projective $\Gamma$-$\Gamma$-bimodule $Q$ such that $N\otimes_\Lambda M\cong (\Lambda\oplus S)\oplus Q\cong \Gamma\oplus Q $ as $\Gamma$-$\Gamma$-bimodules. So $\Lambda$ and $\Gamma$ are stably equivalent of Morita type.
 \end{proof}

{\bf Proof of Theorem \ref{thm:rigidity dimension is invariant under stable equivalences of adjoint type}}. \\
 Let $_AM_B$ and $_BN_A$ be bimodules defining a stable equivalence of Morita type (not necessarily of adjoint type) between $A$ and $B$. Let $_AX$ be a generator in $A\lmod$.
 We claim that $N\otimes_A X$ is a generator in $B\lmod$ and $\gdim\End_A(X)=\gdim\End_B(N\otimes_A X).$
 Indeed, since $N\otimes_A M\cong B\oplus Q$ as $B$-bimodules for
 some projective $B$-bimodule $Q$, it follows that
 ${{_B}B\in \add(N\otimes_A M)}$.
Then ${_A}M$ being projective implies $B\in\add({_B}N)$.
 In other words,  $_BN$ is a projective generator, and thus $_BN\otimes_A X$ is a generator.  By \cite[Theorem 1.1]{LX07}, $\End_A(X)$ and
${\End_B(N\otimes_A X)}$
 are stably equivalent of Morita type.  As global dimensions are preserved by stable equivalences of Morita type, there is an equality $\gdim\End_A(X)=
\gdim\End_B(N\otimes_A X).$

(a) Now, suppose that the pair $(M,N)$ defines a stable equivalence of adjoint type. In other words, the pairs  $(M\otimes_B -, N\otimes_A-)$ and $(N\otimes_A-, M\otimes_B-)$ are adjoint pairs of functors. Further, assume $X$ to be a cogenerator in $A\lmod$. We claim
 that $N\otimes_A X$ is a cogenerator in $B\lmod$ and
$\domdim \End_A(X)= \domdim \End_B(N\otimes_A X).$
Since $_AM\otimes_B-$ is exact with a right adjoint $N\otimes_A-$, injective $A$-modules are sent to injective $B$-modules by $N\otimes_A -$. Similarly, $M\otimes_B-$ sends injective $B$-modules to injective $A$-modules.
Moreover, $N\otimes_A M \otimes_B\D(B)\cong \D(B)\oplus Q\otimes_B\D(B)$.
In particular, $N\otimes_A \D(A_A)$ is a cogenerator in $B\lmod$, and
$N\otimes_A (A\oplus \D(A_A))$ is a generator-cogenerator in $B\lmod$. This
implies that $N\otimes_A X$ is a generator-cogenerator in $B\lmod$.

Since $A$ and $B$ are stably equivalent of adjoint type, it follows from
\cite[Theorem 1.3]{LX07} that $\End_A(X)$ and $\End_B(N\otimes_A X)$
 are stably equivalent of adjoint type, too. As such stable equivalences
 preserve dominant dimension by \cite[Lemma 4.2(2)]{LX07},
 there is an equality $\domdim\End_A(X)=\domdim\End_B(N\otimes_AX)$.
 By the definition of rigidity dimension, $\cf(A)\leq \cf(B)$. Swapping the roles of $A$
 and $B$ yields $\cf(B)\leq\cf(A)$. Thus $\cf(A)=\cf(B)$.

 (b) Under some mild assumptions, stable equivalences of Morita type are of
 adjoint type. Using this, the claims in (b) can be derived from (a) as
 follows:

 $(1)$ If neither $A$ nor $B$ has a simple block,
  then each stable equivalence of Morita type between $A$
  and $B$ is of adjoint type by \cite[Lemma 4.1]{CPX12} and
  \cite[Lemma 4.8]{LX05}. Thus $\cf(A)=\cf(B)$ by (a).
$(2)$ Let $A=A_1\times A_2$ and $B=B_1\times B_2$ such that $A_1$ and $B_1$ are semi-simple and that $A_2$ and $B_2$ have no semi-simple blocks.
  Since $k$ is perfect, the class of  finite-dimensional
  semisimple $k$-algebras coincides with that of finite-dimensional
  separable $k$-algebras. So both $A_1$ and $B_1$ are separable.
  By Lemma \ref{separable}, both $A_2$ and $B_2$ are stably equivalent of Morita type.
  It follows from $(1)$ that $\cf(A_2)=\cf(B_2)$. Since  $\cf(A_1)=\cf(B_1)=\infty$, Proposition \ref{prop:rigidity dimension blockwise} implies $\cf(A)=\cf(B)$.

$(3)$ By \cite[Lemma 4.8]{LX05}, if two self-injective algebras without separable blocks, are stably equivalent of Morita type, then there is a stable equivalence of adjoint type  between them. Separable algebras are semi-simple, hence have infinite rigidity dimension.
  Now, $(3)$ follows from Lemma \ref{separable} and (a)
  together with  Proposition \ref{prop:rigidity dimension blockwise}.
  $\square$

\subsection{Applications to derived equivalences}

Any derived equivalence between self-injective algebras induces a stable
equivalence of Morita type, see \cite[Corollary 2.2]{Ri89}.
The following result is a consequence of Theorem
\ref{thm:rigidity dimension is invariant under stable equivalences of adjoint type}(b)(3). Alternatively, it can be derived from Theorem
\ref{Stable equivalences of self-injective algebras}(I)(3),
since any derived equivalence between self-injective algebras does induce a
triangle equivalence between their stable module categories.

 \begin{cor}\label{thm:rigidity dimension is preserved under derived equivalence}
   Let $A$ and $B$ be self-injective algebras.
   Suppose $A$ and $B$ are derived equivalent. Then $\cf(A)=\cf(B)$.
 \end{cor}

 This can be extended to algebras that are not necessarily self-injective,
 by restricting the class of derived equivalences to certain
 derived equivalences which induce stable equivalences of Morita type.
 These are the almost $\nu$-stable derived equivalences introduced in
 \cite{HX10}.
 Every derived equivalence between two self-injective Artin
 algebras induces an almost $\nu$-stable derived equivalence
 (see \cite[Proposition 3.8]{HX10}). In general, there are still
 many examples of almost $\nu$-stable derived equivalences, for example,
 between Artin algebras constructed in some way from self-injective algebras.

\begin{lem}\textrm{\bf(\cite[Corollary 1.2]{HX13})}\label{syzygy-der}
Let $A$ be a self-injective algebra and $X$ an $A$-module. Then
$\End_A(A\oplus X)$ and $\End_A(A\oplus \Omega_A(X))$ are
almost $\nu$-stable derived equivalent.
\end{lem}

Almost $\nu$-stable derived equivalences in many respects have better
properties than general derived equivalences, for instance in the following
way.

\begin{lem}\textrm{\bf(\cite[Theorem 1.1]{HX10})}\label{almost-nu}
If $A$ and $B$ are almost $\nu$-stable derived equivalent, then they are
stably equivalent of Morita type. In this case, both $A$ and $B$ have the same
global dimension and dominant dimension.
\end{lem}

The following result extends Corollary \ref{thm:rigidity dimension is preserved under derived equivalence} to non-selfinjective algebras.

 \begin{prop} \label{prop:rigidity dimension is invariant under special derived equivalence}
   If $A$ and $B$ are almost $\nu$-stable derived equivalent, then $\cf(A)=\cf(B)$.
 \end{prop}

 \begin{proof}
 Let $A=A_1\times A_2$ and $B=B_1\times B_2$ such that $A_1$ and $B_1$ are semi-simple and that $A_2$ and $B_2$ have no semi-simple blocks. Since derived equivalences preserve semi-simplicity of blocks, both $A_i$ and $B_i$ are derived equivalent for $i=1,2$.  Moreover, $A_2$ and $B_2$ are almost $\nu$-stable derived equivalent. By Lemma \ref{almost-nu}, they are  stably equivalent of Morita type.  It follows from Theorem \ref{thm:rigidity dimension is invariant under stable equivalences of adjoint type}(b)(1) that $\cf(A_2)=\cf(B_2)$.
 Since a semi-simple algebra has infinite rigidity dimension, $\cf(A)=\cf(B)$ by Proposition \ref{prop:rigidity dimension blockwise}.
 \end{proof}

 \begin{cor}\label{omega}
   Let $A$ be a self-injective algebra and $X$ an $A$-module. Then
   $\cf\big(\End_A(A\oplus X)\big)=\cf\big(\End_A(A\oplus\Omega_A(X))\big).$
 \end{cor}

\begin{proof}
  By Lemma \ref{syzygy-der}, $\End_A(A\oplus X)$ and $\End_A(A\oplus\Omega_A(X))$ are almost $\nu$-stable derived equivalent. Now, Corollary \ref{omega} follows from Proposition \ref{prop:rigidity dimension is invariant under special derived equivalence}.
\end{proof}

Finally, we show that some assumptions are needed on derived equivalences to
preserve rigidity dimension.

\begin{example} In general, rigidity dimensions are not preserved under derived
  equivalences. The following is a counterexample.
\\
Let $A$ be the path algebra over $k$ given by the quiver
 $
    \xymatrix{1 \ar[r]^\alpha & 2 \ar[r]^-{\beta}&3},
 $
and let $B$ be the quotient algebra of $A$ modulo the ideal generated by
$\alpha\beta$. Then $A$ and $B$ are derived equivalent via a tilting module of projective dimension one.  Since  $\gdim(A)=1$ and $\gdim(B)=2$, it follows from
Theorem \ref{injective} that $\cf(A)=2$ and $\cf(B)\leq 3$. Note that both the global dimension and the dominant dimension of $\End_B(B\oplus \D(B))$ are equal to 3. This  implies  $\cf(B)\geq 3$, and so $\cf(B)=3$.
Thus $\cf(A)\neq \cf(B)$.
\end{example}

\noindent\textbf{Acknowledgements.} The collaboration began while
the authors visited University of Stuttgart in 2014, and continued
while Steffen Koenig and Kunio Yamagata visited Chinese Academy of
Sciences in 2017. The research was partially supported by National
Science Foundation of China (No.11401397, 11331006, 11471315,
11321101) for the first and the second named author, and also by
JSPS KAKENHI Grant Number 16K05091 for the fifth named author.

\bigskip

\noindent Hongxing Chen,
School of Mathematical Sciences,
BCMIIS Capital Normal University,
100048 Beijing,
P.R.China, \\
Email: \texttt{chx19830818@163.com}
\medskip

\noindent   Ming Fang, Academy of Mathematics and Systems Science,
Chinese Academy of Sciences, 100190 \& School of Mathematical
Sciences, University of Chinese Academy of Sciences, Beijing, 100049\\
Email: \texttt{fming@amss.ac.cn}
\medskip

\noindent Otto Kerner,
Mathematisches Institut,
Heinrich-Heine-Universit\"at,
40225, D\"usseldorf, Germany, \\
Email: \texttt{kerner@math.uni-duesseldorf.de}
\medskip

\noindent   Steffen Koenig,
Institute of Algebra and Number Theory,
  University of Stuttgart,
   Pfaffenwaldring 57,
   70569 Stuttgart, Germany, \\
   Email: \texttt{skoenig@mathematik.uni-stuttgart.de}
\medskip

   \noindent Kunio Yamagata,
Institute of Engineering,
Tokyo University of Agriculture and Technology,
Nakacho 2-24-16,
184-8588, Koganei, Tokyo, Japan, \\
Email: \texttt{yamagata@cc.tuat.ac.jp}
   \medskip

\end{document}